\documentclass[11pt,a4paper,twoside,openright]{article}
\usepackage[english]{ETHDAsfs}

\usepackage{pdfpages}
\usepackage{amsbsy}
\usepackage{amssymb}
\usepackage{graphicx}
\usepackage[longnamesfirst]{natbib}
\usepackage{texab}
\usepackage{amsmath}
\usepackage[utf8]{inputenc}
\usepackage{ae}
\usepackage{enumerate}
\usepackage{relsize}
\usepackage{color} 
\usepackage{color}
\usepackage{tikz}
\newcommand{\indep}{\perp \!\!\! \perp}
\newcommand{\probb}{\text{I\kern-0.15em P}}
\usepackage{caption}
\usepackage{subcaption}
\usepackage{changepage}
\usepackage{fancyvrb}
\usepackage{xcolor}

\usetikzlibrary{shapes,decorations,arrows,calc,arrows.meta,fit,positioning}

\usepackage{listings}
\definecolor{Mygrey}{gray}{0.75}
\definecolor{Cgrey}{gray}{0.4}
\newtheorem{definition}{Definition}
\newtheorem{lemma}{Lemma}

\newtheorem{theorem}{Theorem}
\newtheorem{Coro}{Corollary}
\theoremstyle{definition}

\numberwithin{equation}{subsection}

\graphicspath{{Pictures/}}

\begin{document}
\bibliographystyle{chicago}

\pagenumbering{roman}

\period{2022-2023}
\dasatype{Master's Thesis}
\students{Carlos García Meixide}
\mainreaderprefix{Adviser:}
\mainreader{Prof.\ Dr.\ Peter Bühlmann}

\submissiondate{March 21st 2023}
\title{Causality in time-to-event semiparametric inference}

\cleardoublepage




\newpage
\newpage
\newpage

\cleardoublepage
\phantomsection

\cleardoublepage
\pagenumbering{arabic}

\thispagestyle{plain}
\begin{center}
    \Large
    \textbf{Causal survival embeddings: non-parametric counterfactual inference under censoring}
      \normalsize  

\begin{center}

\vspace{0.4cm}
\begin{tabular}{cc}
\textbf{Carlos García-Meixide} & \textbf{Marcos Matabuena} \\
ETH Zürich& Harvard School of Public Health \\
\texttt{garciac@ethz.ch} & \texttt{mmatabuena@hsph.harvard.edu} \\
\end{tabular}
\end{center}
            
       
    \vspace{0.9cm}

\end{center}

\begin{abstract}

Model-free time-to-event regression under confounding presents challenges due to biases introduced by causal and censoring sampling mechanisms. This phenomenology poses problems for classical non-parametric estimators like Beran's or the k-nearest neighbours algorithm. In this study, we propose a natural framework that leverages the structure of reproducing kernel Hilbert spaces (RKHS) and, specifically, the concept of kernel mean embedding to address these limitations. Our framework has the potential to enable statistical counterfactual modeling, including counterfactual prediction and hypothesis testing, under right-censoring schemes without assumptions on the regression model form. Through simulations and an application to the SPRINT trial, we demonstrate the practical effectiveness of our method, yielding coherent results when compared to parallel analyses in existing literature. We also provide a theoretical analysis of our estimator through an RKHS-valued empirical process. Our approach offers a novel tool for performing counterfactual survival estimation  in observational studies with incomplete information. It can also be complemented by state-of-the-art algorithms based on semi-parametric and parametric models. 
\end{abstract}  

\section{Introduction}

Treatment effect estimation using survival endpoints is of key interest in statistics and biomedical applications. However, the increasing complexity of the processes structuring clinical research during the last decades tend to preclude collecting the correct data regarding a particular clinical question of interest. As a consequence, observational studies are more and more present in scientific research due to technical limitations that make randomization- the gold standard experimental design practice to protect against unmeasured confounding- impossible.

What is more, extensive warnings have been pointed out in the literature abut causal interpretations of hazard ratios (HR) estimated with \cite{cox1972regression}'s Proportional Hazards (PH) model, even under randomized treatment exposures \citep{hernan2010hazards,sten}. As a matter of fact, HRs merge at each instant the differences between arms that arise from treatment effect with those created by selection bias - intuitively, as time goes by, less patients will remain in the control arm if overall mortality risk is different between the two groups (equivalently, when the treatment is effective) leading to a comparison between unbalanced groups. 

This makes clear the point that there is a need for designing new effect measures within survival analysis that
have a causal interpretation and shed light into time dynamics, for instance time-varying treatment effects. Assume that gynaecologists aim to investigate the effect of an implanted medical device, such as a contraceptive method, on time-to-conception. It is reasonable to consider that the implant gradually deteriorates and it will cease to function as time goes by. 
\citet{martinussen2022causality} has shown that a Cox model fails in this setup despite participants being randomized.

Causal treatment effect assertions are phrased via the potential outcome framework, the fundamental paradigm for statistical analysis of observational data - where treatment is not independent of the covariates \citep*{neyman,rubin}. Numerous studies concentrate on computing the average treatment effect (ATE, see \cite{imbens2004}), which determines the discrepancy between the outcome distributions' means. In observational studies, the parameter ATE is interpreted within a framework \citep{pearl2000models} suitable for causal inference as the outcome that would have been observed if the treatment had been randomly assigned.

Most empirical research on treatment effects typically focuses on estimating mean differences. However, there is also a longstanding interest in developing methods to estimate the impact of treatments on the entire outcome distribution. For example, if a specific treatment's impact is only observed in the outcome distribution's variance, the evaluation of average treatment effects will not be informative in the clinical decision-making. 

A straight generalization would be to focus on the difference between the survival functions of potential outcomes directly on the absolute scale:

$$\mathrm{P}\left(\tilde{T^1}>t\right)-\mathrm{P}\left(\tilde{T^0}>t\right)$$

as it is a causally meaningful quantity (the tilde indicates counterfactual) that does not rely on non-collapsible parameters \citep{aalen2015does} nor quantities whose identifiability relies on unstable assumptions (see Section 5 in \cite{martinussen2022causality}). 

Estimating potential outcome distributions directly is straightforward when treatment assignment is random. For instance, in survival analysis it would suffice to fit one \citet*{kaplan1958nonparametric} curve per arm. However, in observational studies (or randomized experiments with imperfect compliance), this type of analysis becomes challenging \citep{imbensiv}. In this line, distributional extensions of the ATE have been considered in the literature through multiple lenses. For example, \cite{abadie} considers a bootstrap strategy to estimate distributional treatment effects while \cite*{cme} base their work on the theory of reproducing kernel Hilbert spaces (RKHSs).

In general, when data was not generated by a randomized control trial, empirical estimators of treatment effects rely either on the imputation of the so-called propensity score; which is the conditional probability of treatment assignment given the covariates \citep*{rosenbaum1983}, either on matching techniques \cite{zubizarreta2012using}, or on combinations of the previous approaches as double robust estimators \cite{ding2018causal}. The fundamental technique here is named \textit{inverse probability of treatment weighted} estimation \citep{imbens2004}, consisting loosely on performing an empirical inner product between the summands of the unweighted estimator times the reciprocals of the estimated propensity scores. Traditional methods for estimating the latter involve parametric approaches like logistic regression, which rely on a model for treatment propensity. However, incorrect specifications of the model can generate extreme weights and make the estimator unreliable. To address this issue, nonparametric techniques have been proposed \citep{lee2010improving}. Nonetheless, large weights may still be unavoidable even when the propensity score model is correctly specified. 

An alternative procedure to avoid weighting consists in decoupling the difference between \textit{realized outcome}- not potential- survival distributions in a particular form inherited from the econometrics literature \citep{oaxaca,blinder}. This decomposition strategy involves two terms: one driven by shifts in the covariate distributions between groups and other accounting for the distributional treatment effect conditional to the treated arm. This provides a formal mechanism to analyze whether differences that arise from the observed outcomes distributions of each arm truly come from effectiveness of a drug, whether they are just due to the probabilistic structure of the population baseline characteristics, or both.

Interpreting the decomposition mentioned above (see Equation \ref{decomp}, Section \ref{sec:pop}) seen within an RKHS provides the primary motivation for the notion of \textit{causal survival embedding}. The main intuition behind the idea is that these objects allow for further investigation of what mechanism is the origin for potential differences that arise between the observational survival functions of two treatment arms (i.e. Kaplan-Meier curves fitted to isolated data coming from each treatment indicator value). It also serves as a departure point for hypothesis testing of covariate distribution shifts across treatment arms. This would involve a pivot that needs the RKHS norm of differences of functions involving our estimator to be computed, constituting in some sense an aggregated measure. However, our estimator is also useful to study differences arising between observational distributions pointwise.

\subsection{Our results and contributions}

We introduce a general non-parametric estimator under right-censoring of counterfactual survival functions based on statistical and machine learning techniques on RKHSs.

\begin{itemize}

\item Our estimation procedure is a model-free approach based on embedding counterfactual distributions in reproducing kernel Hilbert spaces. This extends the prior work of \cite{cme} to handle censored data-generation environments. In contrast to traditional non-parametric survival estimators like the Beran estimator \citep{gonzalez1994asymptotic}, which typically relies on strong smoothness conditions such as differentiability of density functions; our approach does not require these assumptions (only mild conditions on the moments of the kernel function). This makes our method more flexible and applicable to a wider range of scenarios.

\item In the setting of counterfactural inference, our proposal constitutes one of the first strategies in the literature that adjusts for confounding in non-parametric estimation of survival functions.

\item Theoretically, we are able to provide asymptotic behavior guarantees for our estimator and compute its convergence rate by employing techniques from Empirical Process Theory. We utilize these techniques to deduce the Hadamard-differentiability of an operator that takes values in a reproducing kernel Hilbert space. While the Functional Delta Method \citep{van2000asymptotic} is widely known for its application to general operators in Banach spaces, the interplay between the geometry of RKHSs and von Mises calculus remains relatively unexplored in the literature, with only a few researchers delving into this aspect
\cite{Cárcamo2020,matabuena_american}

\item Our procedure is computationally friendly as the main bottlneck is the one classically present in estimating conditional mean embeddings, and linear estimators as  Kaplan-Meier weights, efficiently implemented in well-known software packages such as \verb|survival|.

\item The present work not only advances the field but also paves the way for more advanced formulations of hypothesis testing \cite{gretton2012kernel} and introduces new forms of clustering based on the concept of Maximum Mean Discrepancy \cite{matabuena2022kernel} in the counterfactual setting. Additionally, the flexibility of the RKHS framework enables us to incorporate complex variables, such as medical images or other functional data objects, as predictors. This expanded capability enhances the applicability of our approach to a wider range of domains and data types.

 \item 
 
We demonstrate the potential of our new models through their application in a relevant domain. Specifically, our models provide valuable insights that support the findings in \cite{stensrudathero}, which suggest that the association between treatment-induced diastolic blood pressure and cardiovascular outcomes may be confounded. This corroborates similar findings reported in the literature, as mentioned in the study \citep{beddhu}, which align with our own results. Furthermore, we evaluate the finite properties of our estimator through a comprehensive simulation study, which further validates its effectiveness.

\end{itemize}


\subsection{Other related work}
Different methods exist for treatment effect estimation in presence of censoring. A general procedure that can be found across the literature consists of the following steps. First, the ATE is causally identified without censoring and then a so-called \textit{Censoring Unbiased Transformation} \citep{rubin2007doubly, suzukawa2004unbiased} is used to create a pseud-population from the observed data in which the conditional mean survival time is the same as in the uncensored population. Second, methodology from semiparametric inference adapts the estimators to the censoring mechanisms \citep{tsiatis2006semiparametric}. 
Alternative estimators of treatment effect include standardizing expected outcomes to a given distribution of the confounders (\cite{robins1986new}'s g methods), inverse probability of treatment weighted (IPTW) estimators and doubly robust estimators \citep{ozenne2020estimation}, which combine the two latter lines. 

The utilization of tools from the RKHS framework for right-censored data is relatively limited, with a primary focus on hypothesis testing \cite{matabuena2019energy,rindt2020kernel, tamara}, albeit outside the context of counterfactual inference. There have also been efforts to perform hypothesis testin using RKHSs in other incomplete information schemes such as missing reponse \citep{computmeth} Our methods extend and generalize these works by providing a unified framework to develop new hypothesis testing approaches within the realm of causal inference.

Importantly, \cite{xue2023} balance covariate functions
over an RKHS to avoid directly modelling the propensity score for estimating causal effects.

\subsection{Organisation of the Paper}
The paper is structured as follows. In Section 2 we rigorously introduce the formal elements that constitute the basis of our work by specifying: the fundamental random variables playing a role, which of them are observable and which are not, and how they interact between them to generate incomplete information and notation for their distribution functions. A self-contained description of the parameters of interest is presented in Section 3, accompanied by an opening introducing the notion of counterfactual distributions in survival analysis. Then we define their counterparts in a Hilbert space, leading to the notion of counterfactual mean embedding. Naturally, in Section 4 we thoroughly develop the estimation theory that is needed in our setting, involving M-estimation on a space of functions that themselves take values in another space of functions. The asymptotic properties of our proposal estimator are investigated in Section 5, starting with preliminary definitions needed for its formalization followed by sufficient conditions for consistency and a convergence rate of non-parametric counterfactual inference under censoring. Sections 6 and 7 are concerned with the results, respectively displaying the diminishing behaviour of variability as sample size increases via a simulation study and illustrating the usefulness of our methodology in a real application case related to cardiology. Finally, Section 8 closes the document with a discussion on the consequences of relaxing the censoring assumptions and other concerns regarding open directions. We close the document with a couple of Appendices containing the mathematical proofs for the results of this paper, followed by an empirical check of $\sqrt{n}$ speed of convergence with underlying linear truth. 
\section{Preliminaries}

We start with a collection of random variables in the potential outcomes framework:
$$\{(V^{0} ,V^{1} ,Z ), \quad V \in \{\tilde{T},C,X\}\} $$ 

\begin{itemize}
    \item $\tilde{T}^{0} ,\tilde{T}^{1}  \in (0, + \infty)$ are potential outcomes of survival times of interest.
    \item $C^{0} ,C^{1}   \in (0, + \infty) $ are potential outcomes of censoring times. 
   
    \item $X^0 , X^1  \in \mathbb{R}^p$ are individual vectors of covariates, $p \geq 1$.
     \item $Z  \in \{0,1\}$ are individual treatment assignment indicators.

\end{itemize}

We do not place a tilde over the potential outcomes of censoring times for the sake of simplicity; it is just not needed. 
$F_V$ denotes the distribution function of each random variable $V$. We use standard notation to denote joint and conditional distributions. Next, we define the \textit{realized} survival and censoring times respectively as 
$$T=(1-Z)\tilde{T}^{0}+ Z \tilde{T}^{1}, \quad C=(1-Z){C}^{0}+ Z {C}^{1} $$

The \textit{observed} response is therefore

$$T^*:=\min\{T,C\}$$

We define $T^0$ and $C^0$ as random variables distributed according to $F_{T^0}:=F_{T| Z=0}=F_{\tilde{T}^0| Z=0}$. This function is mathematically relevant because the conditional distribution of times coincides with the conditional distribution of counterfactual times.

The observed covariates are

 $$X=(1-Z)X^{0}+ Z X^{1}$$

 with event indicator 

$$\Delta=(1-Z)1(\tilde{T}^{0}\leq C^0)+Z 1(\tilde{T}^{1}\leq C^1)$$ 
It is worth noting that \begin{align*} (1-Z)\min\{\tilde{T}^0,C^0\}+ Z \min\{\tilde{T}^1,C^1\}=  \min\{(1-Z)\tilde{T}^0 + Z\tilde{T}^1,(1-Z)C^0 + ZC^1\} =\min\{T,C\} \end{align*}

what could be interpreted as commutativity between censoring and realizing.

In practice, we observe an i.i.d sample
$$\{(T^{*}_i,\Delta_i,Z_i,X_i)\}_{i=1}^n  \sim (T^{*},\Delta,Z,X) $$  

which are draws containing incomplete information about the original random variables.

$S=1-F$ denotes survival function in all cases.




\section{Population elements}\label{sec:pop}

\subsection{Counterfactual survival functions}
A key consideration for understanding counterfactual inference is that
$$S_{\tilde{T}^1 \mid Z=1}=S_{Z\tilde{T}^1 + (1-Z)\tilde{T}^1 \mid Z=1}=S_{{T} \mid Z=1}=:S_{{T}^1 }$$

but 

$$S_{\tilde{T}^1 \mid Z=1}\neq S_{\tilde{T}^1}$$

because $\tilde{T}^0$ and $\tilde{T}^1$ may be dependent of $Z$. To guarantee the identifiability of causal
effects from observational data we have to respect the assumption that  the potential outcomes are dependent of the treatment only via the covariates; i.e., there is no hidden confounding. This hypothesis is known as \textit{unconfoundedness} or \textit{ignorability}, which is a common hypothesis in observational studies. We can express it asserting that the joint distribution satisfies the global Markov property with respect to the following undirected graph:

\begin{center}
\begin{tikzpicture}[auto,node distance =1 cm and 1 cm,semithick,
 every loop/.style={},
    state/.style ={circle, draw, minimum width = 0.7 cm},
    point/.style = {circle, draw, inner sep=0.04cm,fill,node contents={}},
    bidirected/.style={Latex-Latex,dashed},
    el/.style = {inner sep=2pt, align=left, sloped}]
 \node[state] (1) {$C^0$};
    \node[state] (2) [right =of 1] {$C^1$};
    \node[state] (3) [right =of 2] {$\tilde T^0$};
        \node[state] (4) [right =of 3] {$\tilde T^1$};

    \node[state] (5) [below =of 4,xshift=-0.0cm,yshift=-0.0cm] {$X$};
            \node[state] (6) [right =of 5] {$Z$};

 \path (1) edge node[above] {} (2);
    \path (3) edge node[above] {} (4);
        \path (1) edge node[above] {} (5);
                \path (2) edge node[above] {} (5);
        \path (3) edge node[above] {} (5);
            \path (4) edge node[above] {} (5);
            \path (5) edge node[above] {} (6);
\end{tikzpicture}

\end{center}

and we term it throughout this paper \textit{conditional exogeneity assumption}, that can be formally expressed as $\tilde{T}^0, \tilde{T}^1 \indep Z \mid X $ and $C^0, C^1 \indep Z \mid X $

Survival functions of potential outcome times conditional on the treatment indicator are of interest because of their involvement in an expression that aims to break down the difference between both \textit{realized} distribution functions for $Z=0,1$. This decomposition serves as one motivation for the foundational work of \citet{cherno13} on counterfactual distributions. The decoupling is the following:

\begin{align}\begin{split}\label{decomp}&S_{{T}^1}(t)-S_{{T}^0}(t)=F_{{T}^0}(t)-F_{{T}^1}(t)=F_{\tilde{T}^0 \mid Z=0}(t)-F_{\tilde{T}^1\mid Z=1}(t)=\\ &\underbrace{F_{\tilde{T}^0 \mid Z=0}(t)-F_{\tilde{T}^0\mid Z=1}(t)}_{(A)}+\underbrace{F_{\tilde{T}^0 \mid Z=1}(t)-F_{\tilde{T}^1\mid Z=1}(t)}_{(B)}\end{split}\end{align}

The equation comprises two terms, (A) and (B), which represent the distributional effect of covariate distributions and the distributional treatment effect on the treated, respectively. The difference between the realized outcome distributions can be attributed to either or both of these terms, and their estimation is valuable in understanding the origin of the difference in observed outcome distributions. If (A) is determined to be small with respect to (B), then the difference between the observed outcome distributions is caused by the distributional difference on the treated (B). Conversely, in the reciprocal configuration, the difference between the observed outcome distributions is due to (A), which is caused by distributional differences between the covariates in each group and not by the effects of the treatment. In the econometrics jargon, (A) quantifies a composition effect due to differences in characteristics and (B) stands for differences in the response structure \citep{cherno13}. 

It is important to note again that $S_{{T}^1}(t)-S_{{T}^0}(t) \neq S_{\tilde{T}^1}(t)-S_{\tilde{T}^0}(t)$. The latter accounts for the effects of treatments 0 and 1, but our approach delves into what is driving the first one to be different. Observed outcome distributions are biased approximations to potential outcome distributions if treatment assignment is not randomized (i.e., if $X$ and $Z$ are not independent).

We now see how these distributions related to potential
outcomes that appear in distributional causal effects are related to the notion of counterfactual distribution, that we define below. In the following, ${F}_{T^0 \mid X^0=x}(\cdot)$ and ${F}_{T^1 \mid X^1=x}(\cdot)$ are the conditional distribution functions that describe the stochastic assignment of survival times to people with characteristics $x$ conditional on $Z=0$ and $Z=1$ respectively. We indistinctly use the relation $S=1-F$.

\begin{definition}[Counterfactual distributions, \cite{cherno13}] 
Whenever support$(F_{X^1})$ $\subseteq$ support$(F_{X^0})$,
$${F}_{T\langle 0 \mid 1\rangle}(\cdot):=\int {F}_{T^0 \mid X^0=x}(\cdot) \mathrm{d} {F}_{X^1}(x)$$
\end{definition}

\begin{lemma}[\cite{cherno13,cme}]
In general, ${S}_{T\langle 0 \mid 0\rangle}={S}_{\tilde{T}^0 \mid Z=0} \text { and } {S}_{T\langle 1 \mid 1\rangle}={S}_{\tilde{T}^1 \mid Z=1} $. Moreover, if conditional exogeneity holds and support$(S_{X^1})=$ support $(S_{X^0})$ then we also have ${S}_{T\langle 0 \mid 1\rangle}={S}_{\tilde{T}^0 \mid Z=1} \text { and } {S}_{T\langle 1 \mid 0\rangle}={S}_{\tilde{T}^1 \mid Z=0} $
\end{lemma}

\begin{proof}
See Lemma 3 and 4 in \cite{cme}.
\end{proof}

If assumptions of Lemma 1 are fulfilled  

$$
(A)  =S_{\langle 0 \mid 1\rangle}(t)-S_{\langle 0 \mid 0\rangle}(t)\\
 =\int F_{T^0 \mid X^0}(t,x) d {F_{X^0}}(x)-\int F_{T^0 \mid X^0}(t,x) d {F_{X^1}} (x)$$

Becoming now clearly visible that (A) is due to a shift between the covariate distributions $F_{X^0}$ and $F_{X^1}$, as the only discrepancy between both integrals in the right hand side is originated by the measures. Meanwhile, as explained, (B) would quantify a treatment effect conditional to the intensive treatment arm.

\subsection{Kernel embeddings}

Let $l: (0,+\infty) \times (0,+\infty) \rightarrow \mathbb{R}$ be a symmetric positive semidefinite function (\textit {kernel}) and $\mathcal{H}$ its associated RKHS \citep{aronszajn1950theory}. We assume for the next couple of definitions that $T^0$ satisfies the integrability condition $\int_{0}^{\infty} \sqrt{l(t, t)} d F_T^0(t)<\infty$, support$(F_{X^1})$ $\subseteq$ support$(F_{X^0})$ and conditional exogeneity. We start with the following definition:

\begin{definition}[Conditional mean embedding, \cite{song2009hilbert}]
$$\mu_{T^0 \mid X^0=x}(\cdot):=\mathbb{E}_{T^0 |X^0}\left[l(T^0, \cdot) \mid X^0=x\right]=\int_0^\infty l(\cdot, t) \mathrm{ d} F_{T^0 \mid X^0=x}(t), \quad x\in \mathbb{R}^p$$
\end{definition}

See \cite{Muandet_2017} for an extensive survey on the interpretation, estimation and properties of kernel- conditional and mean- embeddings. We are now set to reach an important ingredient of our paper:

\begin{definition}[Counterfactual mean embedding, \cite{cme}]\label{deficount}
$$\mu_{T\langle 0 \mid 1\rangle}(\cdot)=\int_{\mathbb{R}^p} \mu_{T^0 \mid X^0={x}}(\cdot) \mathrm{d} F_{X^1}({x}) \in \mathcal{H}$$

\end{definition}

It is easy to see using the iterated expectations lemma and using conditional exogeneity that $\mu_{T\langle 0 \mid 1\rangle}(\cdot)=\int_{0}^{\infty} l(\cdot, t) \mathrm{ d}F_{\langle 0 \mid 1\rangle}(t) $. The previous definitions are reciprocally valid switching 0 by 1 and viceversa. 

\subsection{Interpretation of kernel mean embeddings as depth bands}

\subsubsection{Depth bands}
\begin{definition}
A statistical depth measure is a mapping $D: \mathcal{Y} \times \mathcal{P} \to [0,\infty)$, where $\mathcal{P}$ is the space of probability measures over $\mathcal{Y}$, that satisfies the following properties:

\begin{itemize}
\item Property P-1: Distance invariance of $D$.
\item Property P-2: Maximality of $D$ at the center.
\item Property P-3: Monotonicity of $D$ relative to the deepest point.
\item Property P-4: Upper semi-continuity of $D$ in any function $x \in \mathcal{D}$.
\item Property P-5: Receptivity of $D$ to the convex hull width across the domain.
\item Property P-6: Continuity of $D$ in $\mathcal{P}$.
\end{itemize}
\end{definition}

$h$-integrated depth band measures possess the desirable property of being affine invariant. We introduce the concept of an $h$-depth band functional for any $f\in \mathcal{Y}$, defined as follows:

\begin{equation}
D(f,P_{Y})= \int_{\mathcal{Y}} D_{\kappa_1}\left( \langle f,v \rangle;P_v\right) d \eta (v),
\end{equation}

Here, $D_{\kappa_1}: \mathbb{R}\times \mathcal{P}(\mathbb{R})\to [0,\infty)$ represents a one-dimensional $h$-depth measure using $k_1:[0,\infty)\to [0, \infty)$, $P_v \in \mathcal{P}(\mathbb{R})$ corresponds to the distribution of $\langle f, v \rangle$, where $f\sim P_{Y}$ and $v\in \mathcal{Y}$. The measure $\eta$ is defined on $\mathcal{Y}$ (identified with its dual using the Riesz representation theorem). Importantly, it should be noted that the $h$-depth band remains invariant under affine transformations.
	

		

	Now, we introduce the concept of $h$-bands:
	
	\begin{definition}
	Let $\mathcal{Y}$ be a vector space equipped with a norm $\norm{\cdot}$, $P_{Y}\in \mathcal{P}(\mathcal{Y})$ and $\kappa:[0,\infty)\to [0,\infty)$ be a continuous, non-increasing function with $\kappa(0)>0$ and $\lim_{t\to \infty} \kappa(t)=0$. The $h$-depth of $y \in \mathcal{Y}$ with respect to $P_{Y}$ is defined as 
	
\begin{equation}
D_{\kappa}(y; P_{Y})= \mathbb{E}[\kappa(\norm{y-Y})].
\end{equation} \label{def:bands}
	\end{definition}

\subsubsection{Kernel mean embeddings and integrated depth bands}

A natural connection arises \citep{wynne2021statistical} between $h$-depth and kernel mean embeddings generated by an invariant kernel.

\begin{theorem}
Let $\mathcal{Y}$ be a normed vector space and let $k(x,y)= \kappa(\norm{x-y})$ be a kernel on $\mathcal{Y}$ with $\kappa$ satisfying the conditions of Definition \ref{def:bands}. Then, $D_{\kappa}(y,P)= \phi_{k}P(y)$.
\end{theorem}\label{th:3}

It is natural to ask what conditions are needed on $\kappa$ to ensure that the corresponding kernel $k$ is indeed a kernel. The following theorem contains the key information.

\begin{theorem}
Let $\mathcal{Y}$ be a separable Hilbert space, $\kappa:[0,\infty)\to [0,\infty)$, and $k(x,y)= \kappa(\norm{x-y})$. Then the following are equivalent:
\begin{enumerate}
\item $k$ is a kernel.
\item There exists a finite Borel measure $\mu$ on $[0,\infty)$ such that $k(x,y)= \int_{0}^{\infty}e^{-t^{2}\norm{x-y}^{2}} d\mu(t)$.
\item $\kappa(\sqrt{\cdot})$ is completely monotone.
\end{enumerate}
\end{theorem}

\section{Empirical estimates of causal survival embeddings}

For generality purposes, we denote by $\mathcal{X}$ the covariate space and by $\mathcal{T}$ the target space. Anyway, in our real application case we will use $\mathcal{X}=\mathbb{R}^9$ and $\mathcal{T}=(0,+\infty)$. We first discuss how to estimate
$\mu_{{T^0 \mid X^0}}=\mathbb{E}_{T^0 \mid X^0}\left[l(T^0, \cdot) \mid X^0\right]: \mathcal{X} \longrightarrow \mathcal{H}$ \citep{measure}  because upon obtaining $\hat{\mu}_{T^0 \mid X^0=x}$ (simply done by isolating the data from control group), estimating counterfactual mean embeddings is reduced to taking averages with respect to the covariates in the treatment group: $\hat{\mu}_{T\langle 0 \mid 1\rangle}:=\frac{1}{m} \sum_{j=1}^m \hat{\mu}_{T^0 \mid X^0={X^1_j}}$, as suggested by Definition \ref{deficount}. We insist in the fact that for estimation of the \textit{conditional} mean embedding we use data in the \textit{control} group. In this section we focus on $T^0$ and $X^0$ but, again, the same theory holds replacing $0$ by $1$ without loss of generality when it comes to estimate $\mu_{T\langle 1 \mid 0\rangle}$. We start by noticing that the map $x \mapsto {\mu}_{T^0 \mid X^0=x}$ takes values in the Hilbert space $\mathcal{H}$. This motivates the following definition:

\begin{definition}[Vector-valued RKHS, \cite{carmeli2006vector}]
An $\mathcal{H}$-valued RKHS on $\mathcal{X}$ is a Hilbert space $\mathcal{F}$ such that 1) the elements of $\mathcal{F}$ are functions $\mathcal{X} \rightarrow \mathcal{H}$; 2) for all $x \in \mathcal{X}, \exists C_x>0$ such that $\|F(x)\|_{\mathcal{H}} \leq C_x\|F\|_{\mathcal{F}} \textrm{ for all } F \in \mathcal{F}$.
\end{definition}

In the traditional framework of RKHSs formed by real valued functions, a very useful aspect is that it is possible to evaluate functions belonging to the space by making inner products times very special elements therein: the collection $\{k(\cdot, x): x \in \mathcal{X}\}$, in virtue of Riesz's Representation Theorem. $k$ is the so-called \textit{kernel} function uniquely determining $\mathcal{H}$. Looking for a surrogate of the notion of kernel in $\mathcal{H}$-valued RKHSs we arrive to the following definition. We call $\mathcal{L}(\mathcal{H})$ the space of bounded linear operators from $\mathcal{H}$ to $\mathcal{H}$.

\begin{definition}[$\mathcal{H}$-kernel, \cite{carmeli2006vector}]
 
A $\mathcal{H}$-kernel of positive type on $\mathcal{X} \times \mathcal{X}$ is a map $\Gamma: \mathcal{X} \times \mathcal{X} \rightarrow \mathcal{L}(\mathcal{H})$ such that $\forall N \in \mathbb{N}, \forall x_1, \ldots, x_N \in \mathcal{X}$ and $\forall c_1, \ldots, c_N \in \mathbb{R}, \sum_{i, j=1}^N c_i c_j\left\langle\Gamma\left(x_j, x_i\right) (h), h\right\rangle_{\mathcal{H}} \geq 0 \quad \forall h \in \mathcal{H}$.

\end{definition}
If $\Gamma$ is an $\mathcal{H}$-kernel in the sense of the previous definition, there exists a unique (up to isometry) RKHS, with $\Gamma$ as its reproducing kernel \citep{pontil}, satisfying: 1) for any $x, x^{\prime} \in \mathcal{X}, h, h^{\prime} \in \mathcal{H}$ and $F \in \mathcal{F}$, $\langle F(x), h\rangle_{\mathcal{H}}=\langle F, \Gamma(\cdot, x) (h)\rangle_{\mathcal{F}}$ and 2) $\left\langle h, \Gamma\left(x, x^{\prime}\right)\left(h^{\prime}\right)\right\rangle_{\mathcal{H}}=\left\langle\Gamma(\cdot,x)(h), \Gamma\left(\cdot, x^{\prime}\right)\left(h^{\prime}\right)\right\rangle_{\mathcal{F}}$

Now we can pose the estimation of conditional mean embeddings as risk minimization of the theoretical loss \citep{grune}:  $$\tilde R(F)=\mathbb{E}_{X^0}\left[\left\|\mu_{{T^0 \mid X^0}}(X^0)-F(X^0)\right\|_{\mathcal{H}}^2\right], \quad F \in \mathcal{F}$$ 

where $\mathcal{F}$ is a vector-valued RKHS of functions $\mathcal{X} \rightarrow \mathcal{H}$. For simplicity, we endow $\mathcal{F}$ with a kernel $\Gamma\left(x, x^{\prime}\right)=k\left(x, x^{\prime}\right)$ Id, where $k$ is a scalar kernel on $\mathcal{X}$ and Id: $ \mathcal{H} \rightarrow \mathcal{H}$ is the identity map on $\mathcal{H}$. We have in virtue of generalised conditional Jensen’s inequality \citep{perlman1974jensen} and iterated expectations lemma:

$\begin{aligned}
\tilde R(F) & =\mathbb{E}_{X^0}\left[\left\|\mathbb{E}_{T^0 \mid X^0}\left[l(T^0, \cdot)-F(X^0) \mid X^0\right]\right\|_{\mathcal{H}}^2\right] \leq \mathbb{E}_{X^0} \mathbb{E}_{T^0 \mid X^0}\left[\left\|l(T^0, \cdot)-F(X^0)\right\|_{\mathcal{H}}^2 \mid X^0\right] \\
& =\mathbb{E}_{T^0X^0}\left[\left\|l(T^0, \cdot)-F(X^0)\right\|_{\mathcal{H}}^2\right]=: R(F)
\end{aligned}$

$R(F)$ acts as a surrogate theoretical risk that admits an empirical version under right-censoring.

Now the problem is that we do not have access to a sample from the joint distribution of $(T_0,X_0)$ that would allow us to estimate the expectation involved by $R(F)$ because of censoring: we instead observe data from $\operatorname{min}\{T^0,C^0\}$. Let us further develop the measure with respect to which the expectation in $R(F)$ is taken:
\begingroup
\addtolength{\jot}{2em}
\begin{align*}
 dF_{T^0X^0}(t,x)& =P(T^0 \in dt,X^0 \in dx )= P(T \in dt,X \in dx | Z = 0 ) = \\
 &=\frac{ P(T \in dt,X \in dx | Z = 0 ) P(Z = 0 )}{P(Z = 0 )} = \\
&=  \frac{ P(T \in dt,X \in dx , Z = 0 ) P(\Delta=1 | T =t,X =x , Z = 0 ) }{P(Z = 0 ) P(\Delta=1 | T =t,X=x , Z = 0 )} =\\ 
&= \frac{  P(\Delta=1 , T \in dt,X \in dx , Z = 0 ) }{P(Z = 0 )P(\Delta=1 | T =t,X =x , Z = 0 )} = \frac{  P(\Delta=1 ,T \in dt,X \in dx  | Z = 0 ) }{P(\Delta=1 |  T =t,X =x , Z = 0 )}= \\
&=\frac{ dF_{0}^{(*)}(t,x)  }{G_0(t,x)}
\end{align*}

where $G_0(t,x)=P(\Delta=1| T=t,X=x, Z = 0)$ is the conditional probability that an observation is uncensored given that the event time is $t$ and the covariates are $x$ in the control population \\ and $F_0^{(*)}(t,x) = P( \Delta =1,T \leq t,X \leq x|Z = 0)$ is the law of uncensored observations in the control population \citep*{stute1996distributional,gerds}. 

Note that if we assume 
 $C \indep T | Z $ and $\Delta \indep X | T,Z$ then 

 $G_0(t,x)=P(\Delta=1| T=t,X=x, Z = 0)=P(\Delta=1| T=t, Z = 0)=P(C>t| Z = 0)$ 

 and therefore $G_0(t,x)=G_0(t)$ equals $1-$ the marginal law of censoring times conditional to $Z=0$.

Let $\left(X_1, T_1^{*}\right), \ldots,\left(X_n, T_n^{*}\right)$ be i.i.d. observations from the control group $Z=0$. By plugging in an estimate $\hat G_0(t,x)$ and the empirical measure   $$d\hat F_0^{(*)}(t,x)=\frac{1}{n} \sum_{i=1}^{n}\Delta_i \delta_{T^*_i}(t)\delta_{X_i}(x)$$ we arrive to a regularized empirical risk minimization problem:

$$\hat{R}_ {\varepsilon,n}(F):=\frac{1}{n}\sum_{i=1}^n\frac{\Delta_i}{\widehat G_0 (T^*_i,X_i)}\left\|l\left(T^*_i, \cdot\right)-F\left(X_i\right)\right\|_{\mathcal{H}}^2+\varepsilon\|F\|_\mathcal{F}^2$$

$$W_i:=\frac{\Delta_i}{\widehat G_0 (T^*_i,X_i)}$$

We denote its minimizer by $\hat{\mu}_{\varepsilon, n}$,
$$
\hat{\mu}_{\varepsilon, n}:=\underset{F \in \mathcal{F}}{\operatorname{argmin}} \phantom{s}\widehat{R}_{\varepsilon, n}(F) .
$$
This is the final estimator for the {conditional} mean embedding. 


\begin{lemma}\label{l2}
A minimizer of the empirical risk $\hat{R}_ {\varepsilon}(F)$ is unique and can be expressed as $\sum_{j=1}^n \Gamma\left(\cdot, x_i\right)\left(c_i\right)$ where the coefficients $\left\{c_j: j=1, \ldots, n\right\} \subseteq \mathcal{H}$ are the unique solution of the linear equations $\sum_{j=1}^n\left(W_i\Gamma\left(z_i, z_j\right)+n \varepsilon \delta_{i j}\right)\left(c_j\right)=W_ih_i, i=1, \ldots, n$.
\end{lemma}

\begin{proof}
See Appendix.
\end{proof}

Choosing $\Gamma\left(x, x^{\prime}\right)=k\left(x, x^{\prime}\right) \mathrm{Id}$ (see \cite{grune} for more details on why this is a sensible election) we conclude
$$W H=(WK+ n\varepsilon I ) C \Longleftrightarrow  C=(WK + n\varepsilon I ) ^{-1}WH$$

where ${K}_{ij}=k(X_i,X_j)$ $W=\textrm{diag}(W_1, \ldots, W_n)$, $H=(h_1\ldots h_n)'$, $C=(c_1\ldots c_n)'$. 

Now the \textit{conditional} mean embedding evaluated on the covariates of the treated sample $(X^1_1, \ldots,X^1_m )$ is $(\hat F(X^1_1) \ldots \hat F(X^1_m))=(\sum_{j=1}^nk(X^1_1 ,X_j) c_j\ldots\sum_{j=1}^nk(X^1_m ,X_j) c_j) = C'\Tilde{K}$

where $\Tilde{K}_{ij}=\Gamma(X_i,X^1_j)$

The \textit{counterfactual} mean embedding is computed by taking the average of the previous row: $\hat \mu_{T\langle 0 \mid 1\rangle}(\cdot) = C'\Tilde{K} 1_m $ where $1_m$ is a vector of all ones divided by $m$. 

By recovering the expression of $C$ previously derived we have a closed expression for the \textit{counterfactual} mean embedding estimator

$$\hat\mu_{T\langle 0 \mid 1\rangle}(\cdot)= ((WK + n\varepsilon I ) ^{-1}WH)'\Tilde{K} 1_m = H' W(KW + n\varepsilon I ) ^{-1}\Tilde{K} 1_m$$

and its row-shaped version (visually, resembles better to a function of time) is
$$\hat \mu'_{T\langle 0 \mid 1\rangle}(\cdot)= 1'_m \Tilde{K}'(WK + n\varepsilon I ) ^{-1}W H  $$

It is important to bear in mind that $H=(l(T^*_1,\cdot),\cdots,l(T^*_n,\cdot))'$.
We can always evaluate $H_{ij}=l(T^*_i,t_j)$ on a grid time-points $t_1,\ldots, t_N$. 

\section{Asymptotics of causal survival embeddings}
\subsection{Population and empirical covariance operators}
This section comprises the main theoretical contribution of our work. Let us get started by a couple of definitions needed to reexpress parameters and their estimators in a more convenient way regarding proofs. 
\begin{definition}[\cite{fukumizubayes}] Let $\mathcal{C}_{T X}: \mathcal{G} \rightarrow \mathcal{H}$ be the covariance operator of the random variables $X^0$ and $T^0$ defined as 
$$
\mathcal{C}_{T X} f=\int l(\cdot,  {t}) f( {x}) d F_{X^0 T^0}( {x}, {t})=\mathbb{E}_{X^0 T^0}\left[l\left(\cdot, T^0\right) f\left(X^0\right)\right], \quad f \in \mathcal{G}
$$
\end{definition}

substituting the measure $dF_{X^0 T^0} = \frac{ dF_{0}^{(*)}  }{G_0}$ by the empirical counterparts $\hat F_{0}^{(*)}$ and $\hat G_0$ we obtain
\begin{definition}[Adapted from \citet*{cme}]Let $\left(X_1, T_1^{*}\right), \ldots,\left(X_n, T_n^{*}\right)$ be i.i.d. observations from the control group $Z=0$. We define: 
$$\widehat{\mathcal{C}}^*_{X X} f:=\frac{1}{n} \sum_{i=1}^n W_i k\left(\cdot, X_i\right) f\left({X}_i\right), \quad \widehat{\mathcal{C}}^*_{TX} f=\frac{1}{n} \sum_{i=1}^n W_i l\left(\cdot, T^*_i\right) f\left(X_i\right), \quad f \in \mathcal{G}$$
\end{definition}
The following result shows that we can write $\hat{\mu}_{T\langle 0 \mid 1\rangle}$ using the empirical covariance operators. 

\begin{lemma}\label{l3}
Let $\hat{\mu}_{X_1}$ the kernel mean embedding estimated with the sample covariates from the treated population. Then we have
$$
\hat{\mu}_{T\langle 0 \mid 1\rangle}=\widehat{\mathcal{C}}^*_{TX}\left(\widehat{\mathcal{C}}^*_{X X}+\varepsilon I\right)^{-1} \hat{\mu}_{X_1} .
$$

\end{lemma}
\begin{proof}
See Appendix. 
\end{proof}

\subsection{Assumptions}

In the following, we introduce the assumptions needed for establishing consistency of our proposed estimator. 
\begin{enumerate}
\item $\sup _{x \in \mathcal{X}} k(x, x)<\infty$ and $\sup _{t \in \mathcal{T}} \l(t, t)<\infty$

This assumption is satisfied by Gaussian kernels and helps in conjunction with the following general inequality for RKHSs. Let us suppose that $f \in \mathcal{G}$. Then for $x \in \mathcal{X}$

$$ f(x)=\left \langle k(\cdot,x), f \right \rangle_{\mathcal{H}} \leq  \left\| k(\cdot,x) \right\|_{\mathcal{H}} \left\|  f \right\|_{\mathcal{H}}   $$

in virtue of Cauchy-Schwartz inequality. Now noting that $\left\| k(\cdot,x) \right \|^2_{\mathcal{H}} = \left \langle k(\cdot,x), k(\cdot,x) \right \rangle_{\mathcal{H}} = k(x,x)$, we finally have
$$ f(x) \leq  \sqrt{ k(x,x) } \left\|  f \right\|_{\mathcal{H}}   $$

and therefore

$$ \left\|f\right \|_{\infty} \leq  \operatorname{sup}_{x \in \mathcal{X}}\left\lvert k(x,x) \right \rvert  \left\|  f \right\|_{\mathcal{H}}   $$

As a particular case 

$$ k(x,x') \leq  \sqrt{ k(x,x) } \sqrt{ k(x',x') }   $$

Moreover, as all probability measures are finite we have ensured that $k$ is integrable with respect to any probability measure in virtue of Hölder's inequality.

\item $\text {The RKHS } \mathcal{H} \text { of } k \text { is dense in } L_2\left(F_{X_0}\right) \text {. }$This is also satisfied by Gaussian kernels \citep{steinwart2008support}. 
\item The distribution $F_{X_1}$ is absolutely continuous with respect to $F_{X_0}$ with the Radon-Nikodym derivative $g:=\mathrm{d}F_{X_1} / \mathrm{d} F_{X_0}$ satisfying $g \in L_2\left(F_{X_0}\right)$. By this we are expressing formally that the marginal density functions of $F_{X_0}$ and $F_{X_1}$ should not be very different. It also implies the support equality condition used throughout Section \ref{sec:pop}. 

\item $\left(T_1^{*},\Delta_1,0,X_1\right), \ldots,\left(T_n^{*},\Delta_n,0,X_n\right)$ are i.i.d. observations from the control group, and $X_1^1 \ldots, X_m^1 $ are i.i.d. observations of the random variable $X^1$.
\item $C \indep T | Z $ (independence) and $\Delta \indep X | T,Z$ (conditional independence of the censoring indicator and the covariates given the realized time).
This automatically implies 
$$G_0(t,x)=P(C>t | T=t,X=x,Z=z)=P(C>t | T=t)=P(C>t)=:G_0(t)$$ In this case, it is possible to estimate $G_0(t)$ using the marginal reverse Kaplan-Meier estimator- flipping the event indicators and using the canonical Kaplan-Meier estimator \citep{gill1980censoring}. See \cite{stute1993, stute1996distributional} for further comments on these couple of assumptions. 
\item $$\frac{1}{G_0^2}\textrm{ and }\frac{1}{\hat{G}_0^2} < \infty$$
This ensures that population and empirical covariance operators are well defined as Bochner integrals \citep{dinculeanu2000vector}. 

\subsection{Consistency and convergence rate}
Our main theoretical contribution is the convergence rate of the stochastic error in RKHS norm in Theorem \ref{rate_stoch}. Once established, we complement our finding with the literature aiming to prove consistency in Corollary \ref{cons} and find the final convergence rate in Corollary \ref{final_rate}.

\end{enumerate}

\begin{theorem} \label{rate_stoch} (Convergence rate of the stochastic error)
Consider the causal survival embedding estimator $\hat{\mu}_{T\langle 0 \mid 1\rangle}$. Suppose that conditions i.) to vi.) (ii.) is optional) hold. Then we have for the stochastic error
$$\left\|\widehat{\mathcal{C}^*}_{T X}\left(\widehat{\mathcal{C}^*}_{X X}+\varepsilon_n I\right)^{-1} \widehat{\mu}_{X_1}-\mathcal{C}_{T X}\left( \mathcal{C}_{X X}+\varepsilon_n I\right)^{-1} \mu_{X_1} \right\|_{\mathcal{H}} = O_p\left(n^{-1/2}\varepsilon_n^{-1} \right) $$

\end{theorem}

\begin{proof}
We start with the same breakdown as in proof of Theorem 11 in \citet*{fukumizubayes}: \begin{equation*}
\begin{split}
& \left\|\widehat{\mathcal{C}^*}_{T X}\left(\widehat{\mathcal{C}^*}_{X X}+\varepsilon_n I\right)^{-1} \widehat{\mu}_{X_1}-\mathcal{C}_{T X}\left( \mathcal{C}_{X X}+\varepsilon_n I\right)^{-1} \mu_{X_1} \right\|_{\mathcal{H}} \leq \\ \leq &\left\|\widehat{\mathcal{C}^*}_{T X}\left(\widehat{\mathcal{C}}_{X X}+\varepsilon_n I\right)^{-1}\left(\widehat{\mu}_{X_1}-\mu_{X_1}\right)\right\|_{\mathcal{H}}  \quad : \quad  (A) \\
+ &\left\|\left(\widehat{\mathcal{C}^*}_{T X}-\mathcal{C}_{T X}\right)\left(\mathcal{C}_{X X}+\varepsilon_n I\right)^{-1} \mu_{X_1}\right\|_{\mathcal{H}} \quad : \quad  (B) \\
+ & \left\|\widehat{\mathcal{C}^*}_{TX}\left(\widehat{\mathcal{C}^*}_{X X}+\varepsilon_n I\right)^{-1}\left(\mathcal{C}_{XX}-\widehat{\mathcal{C}^*}_{X X}\right)\left(\mathcal{C}_{X X}+\varepsilon_n I\right)^{-1} \mu_{X_1}\right\|_{\mathcal{H}} \quad  : \quad   (C)
\end{split}
\end{equation*}

(A): From \cite{cme} we have that 
$$
\textrm{(A)} = O_p\left(\varepsilon_n^{-1 / 2} n^{-1 / 2}\right)
$$

as it can be seen to rely on weak convergence of uncensored kernel mean embeddings at speed $\frac{1}{\sqrt n}$ \citep{ledoux1991probability, berlinet2011reproducing} and on applying Theorem 1 in \cite{baker1973joint} to $\frac{ d\hat F_{0}^{(*)}  }{\hat G_0}$.

(B): using Lemma 24 in \citet{cme}

\begin{align*}
\left\|\left(\widehat{\mathcal{C}^*}_{T X}-\mathcal{C}_{T X}\right)\left(\mathcal{C}_{X X}+\varepsilon_n I\right)^{-1} \mu_{X_1}\right\|_{\mathcal{H}} & \leq\left\|\widehat{\mathcal{C}^*}_{T X}-\mathcal{C}_{T X}\right\|\left\|\left(\mathcal{C}_{X X}+\varepsilon_n I\right)^{-1} \mu_{X_1}\right\|_{\mathcal{G}} \\
& \leq\left\|\widehat{\mathcal{C}^*}_{T X}-\mathcal{C}_{T X}\right\| \cdot O_p\left(\varepsilon_n^{-1/2} \right) 
\end{align*}

(C): proceeding as in \citet{cme}
$$
\textrm{(C)} =\left\|\widehat{\mathcal{C}}^*_{X X}-\mathcal{C}_{X X}\right\| \cdot O_p\left(\varepsilon_n^{-1} \right)
$$

Let $\varepsilon_n>0$ be a regularization constant. Then if $\varepsilon_n \rightarrow 0$ and $n^{1 / 2} \varepsilon_n \rightarrow \infty$ as $n \rightarrow \infty$, we have consistency provided that we show the tight uniform bounds 
\begin{equation}\label{t1}
    \left\|\widehat{\mathcal{C}^*}_{XX}-\mathcal{C}_{X X}\right\|=O_p\left(n^{-1 / 2}\right)
\end{equation}

\begin{equation}\label{t2}
\left\|\widehat{\mathcal{C}^*}_{T X}-\mathcal{C}_{T X}\right\|=O_p\left(n^{-1 / 2}\right)
\end{equation}

and the term with the the slowest rate would be (C). We will have into account that $\|\cdot\| \leq\|\cdot\|_{H S}$.

\begin{lemma}

Define $K_i=k(\cdot,X_i)-\mu_{X^0}$, $L_i=l(\cdot,T^*_i)-\mu_{T^0}$, $K(X^0)=k(\cdot,X^0)-\mu_{X^0}$, $L(T^0)=l(\cdot,T^0)-\mu_{T^0}$ where $\mu_{X^0}$ and $\mu_{T^0}$ are the marginal kernel mean embeddings $E_{X^0}\left[k(\cdot, X^0)\right]$ and $E_{T^0}\left[k(\cdot, T^0)\right]$. Then we have:

$$\|\widehat{\mathcal{C}^*}_{T X}-\mathcal{C}_{T X}\|_{H S}^2= \left\|\frac{1}{n} \sum_{i=1}^nW_i\left(K_i-\frac{1}{n} \sum_{j=1}^nW_j K_j\right)\left(L_i-\frac{1}{n} \sum_{j=1}^n W_jL_j\right)-E[K(X^0) L(T^0)]\right\|_{\mathcal{G}\otimes\mathcal{H}}^2$$

\end{lemma}

\begin{proof} Direct adaptation of \cite{fukumizu07a}. 
\end{proof}
Deriving the following inequality in Lemma \ref{lem:ineq} is more involved however

\begin{lemma}\label{lem:ineq}

\begin{align*}& \|\widehat{\mathcal{C}^*}_{T X}-\mathcal{C}_{T X}\|_{H S}  \leq \\ &\leq \left\|\frac{1}{n} \sum_{i=1}^nW_iK_iL_i-E[K(X^0) L(T^0)]\right\|_{\mathcal{G}\otimes\mathcal{H}} +\left|2 -\frac{1}{n}\sum_{i=1}^nW_i\right|\left\|\frac{1}{n}\sum_{i=1}^nW_i K_i\right\|_{\mathcal{G}}\left\|\frac{1}{n}\sum_{i=1}^nW_iL_i \right\|_{\mathcal{H}} 
\end{align*}
\end{lemma}
\begin{proof}
See Appendix.
\end{proof}

Let us denote for simplicity of notation $\mu_{X^0}=\mu_0$. Having a closer look at the term $\left\|\frac{1}{n}\sum_{i=1}^nW_i K_i\right\|_{\mathcal{G}}$ in right hand side of Lemma \ref{lem:ineq}
\begin{align*}\frac{1}{n}\sum_{i=1}^nW_iK_i&=\frac{1}{n}\sum_{i=1}^nW_i(k(\cdot,X_i)-\mu_0)=\frac{1}{n}\sum_{i=1}^n(W_ik(\cdot,X_i)-W_i\mu_0)\\&=\frac{1}{n}\sum_{i=1}^nW_ik(\cdot,X_i)-\mu_0+\mu_0-\mu_0\left(\frac{1}{n}\sum_{i=1}^nW_i\right)=\\
&=\left({\frac{1}{n}\sum_{i=1}^nW_ik(\cdot,X_i)-\mu_0}\right)+\mu_0\left(1-\frac{1}{n}\sum_{i=1}^nW_i\right)
\end{align*}

Furthermore,

\begin{align*}
{\frac{1}{n}\sum_{i=1}^nW_ik(\cdot,X_i)-\mu_0}=\int _\mathcal{X}k(\cdot,X^0)\frac{d\hat F_0^{(*)}}{\hat{G_0}} -\int _\mathcal{X}k(\cdot,X^0)\frac{d F_0^{(*)}}{G_0}=:\nu(\hat F_0^{(*)},\hat{G}_0)- \nu(F_0^{(*)},{G_0}) \in \mathcal{G}
\end{align*}

It is important to note that $\nu$ is an operator taking values in a Hilbert space and showing its Hadamard-differentiability is not straightforward. Let, tor $n \geq 1$,
$S_n=\sum_{i=1}^n k\left(\cdot, X_i\right)$ and
$\Lambda_n=\sqrt{n}\left(\frac{S_n}{n}-\mathcal{I}_\mu\right) .$
Since $\mathcal{I}_\mu=\int K\left(\cdot, X_0\right) d F_{X^0}=E\left(K\left(\cdot, X_0\right)\right)$, one could prove by using the Hilbert space version of the Central Limit Theorem that the sequence $\left(\Lambda_n\right)_{n \geq 1}$ converges weakly to a centered gaussian variable \citep*{ledoux1991probability}. The elements preventing us from proceeding this way are the $W_i$, which are breaking the i.i.d. assumption needed by this CLT. 

First, it is known that $\sqrt{n}(\hat{G}_0-G_0)$ converges weakly in $D[0, \tau]$ to a tight, mean zero Gaussian process (\cite{fleming2011counting, andersen2012statistical}). Second, by Donsker's theorem, $\sqrt n(\hat F_0^{(*)}-F_0^{(*)})$ also converges weakly to a tight, mean zero Gaussian process- as a reminder, $d\hat F_0^{(*)}(t,x)=\frac{1}{n} \sum_{i=1}^{n}\Delta_i \delta_{T^*_i}(t)\delta_{X_i}(x)$ is just the empirical measure of the uncensored observations on the arm with $Z_i=0$

We now proceed to show Hadamard-differentiability of $\nu$ for $\mathcal{X}=\mathbb{R}$. The following definitions are taken from \cite{van2000asymptotic} sections 18.6 and 20.3
\begin{definition}
Let $T=[a, b]$ be an interval in the extended real line. We denote by $C[a, b]$ the set of all continuous functions $z:[a, b] \mapsto \mathbb{R}$ and by $D[a, b]$ the set of all functions $z:[a, b] \mapsto \mathbb{R}$ that are right continuous and whose limits from the left exist everywhere in $[a, b]$. (The functions in $D[a, b]$ are called cadlag: continue à droite, limites à gauche.) It can be shown that $C[a, b] \subset D[a, b] \subset \ell^{\infty}[a, b]$. We always equip the spaces $C[a, b]$ and $D[a, b]$ with the uniform norm $\|z\|_T$, which they "inherit" from $\ell^{\infty}[a, b]$
\end{definition}

The space $D[a, b]$ is referred to here as the Skorohod space and the set $B V_M[a, b]$ is the set of all cadlag functions $z:[a, b] \mapsto[-M, M] \subset \mathbb{R}$ of variation bounded by $M$. We also define:

$B V^1_M[a, b]=\{B\in B V_M[a, b] : x \mapsto k(x,x) \in L^1(B)  \}$

$D^2[a,b]=\{A\in D[a, b] : A \in L^2(B)  \textrm{ for all } B \in B V^1_M[a, b]\}$

We need to restrict our operator to
${D}_M \equiv D^{2}[-\infty, \infty]\times$ $B V^{1}_M[-\infty, \infty]$ for existence of Bochner integrals, see Theorem 105 in \cite{berlinet2011reproducing}. Nevertheless, thanks to assumptions i.) and vi.) this is always the case as far as we operate on $D[a, b] \times B V_M[a, b]$.

\begin{lemma}
Let $\mathcal{H}$ be an RKHS of functions $f : \mathbb{R} \longrightarrow \mathbb{R}$ with reproducing kernel $k$. Then the operator $\left(A, B\right) \mapsto \int k(\cdot,x)A(x) d B(x) \in \mathcal{H}$ is Hadamard-differentiable from the domain ${D}_M \equiv D^{2}[-\infty, \infty]\times B V^{1}_M[-\infty, \infty] \subset D[-\infty, \infty] \times D[-\infty, \infty]$ into $(\mathcal{H},\sqrt{\langle\cdot, \cdot \rangle_{\mathcal{H}} })$ at every pair of functions of bounded variation $\left(A, B\right)$.
\end{lemma}

\begin{proof}

We set as a candidate $\psi_{A, B}^{\prime}(\alpha, \beta)(\cdot)=\int k(\cdot,x)A(x)d \beta(x)+\int k(\cdot,x) \alpha(x) d B(x),$ \\ for $ (\alpha,\beta) \in D_M$.

We will use the fact that $\norm{k(\cdot,x)}_{\mathcal{H}}^2=\langle k(\cdot,x),k(\cdot,x)\rangle=k(x,x) $. 

 For sequences $t_n \rightarrow 0$ in $\mathbb{R}$,  $\alpha_n \rightarrow \alpha$, and $\beta_n \rightarrow \beta$ in $D^{2}[-\infty, \infty]$ and $B V^{1}_M[-\infty, \infty]$ respectively, define $A_n \equiv A+t_n \alpha_n$ and $B_n \equiv B+t_n \beta_n$. Since we require that $\left(A_n, B_n\right) \in$ $D_M$, we know that the total variation of $B_n$ is bounded by $M$. Consider first the derivative of $\psi$, and note that 
 
 $$\begin{aligned} & \norm{\frac{\int k(\cdot,x) A_n(x) d B_n(x)-\int k(\cdot,x) A (x)d B(x)}{t_n}-\psi_{A, B}^{\prime}\left(\alpha_n, \beta_n\right)}_{\mathcal{H}}= \\ & 
  \norm{\int k(\cdot,x) \alpha_n(x) d\left(B_n-B\right)(x)}_{\mathcal{H}}
  = \\ & 
  \norm{\int k(\cdot,x) \alpha(x) d\left(B_n-B\right)(x)+\int k(\cdot,x)\left(\alpha_n(x)-\alpha(x)\right) d\left(B_n-B\right) (x)}_{\mathcal{H}}\leq \\ & 
   \int \norm{k(\cdot,x)}_{\mathcal{H}} |\alpha (x)|d\left(B_n-B\right)(x)+\int \norm{k(\cdot,x)}_{\mathcal{H}}|\alpha_n(x)-\alpha(x)|d\left(B_n-B\right) (x)=\\ & 
     \int \sqrt{k(x,x)}| \alpha (x)|d\left(B_n-B\right)(x)+\int \sqrt{k(x,x)}\left(|\alpha_n(x)- \alpha(x)|\right) d\left(B_n-B\right) (x) \equiv \textrm{(1) + (2)}
   \end{aligned}$$



    Because we assumed that $k$ is bounded, (2) converges to zero as since both $B_n$ and $B$ have total variation bounded by $M$ and $k$ is bounded.

For convergence of (1) to zero as $t_n \longrightarrow 0$ we follow the same argument as in \cite{van2000asymptotic} Lemma 20.10. with $\phi$ therein the identity map. 
 
Since the map $(\alpha, \beta) \mapsto \psi_{A, B}^{\prime}(\alpha, \beta)$ is continuous and linear, the desired Hadamard differentiability of $\psi$ will follow because (1) and (2) converge to zero. 

\end{proof}
Our operator $\nu$ was defined for $(A,B) \in D_M$ as
$$\nu: (A, B) \mapsto\left(A, \frac{1}{B}\right) \mapsto \int_{\mathcal{X}}k(\cdot,X) \frac{1}{B(X)} d A(X)$$

We can assert that $\sqrt{n}\left(\nu(\hat F_0^{(*)},\hat{G}_0)- \nu(F_0^{(*)},{G_0})\right)$
converges weakly to a process in a Polish RKHS in virtue of the chain rule of Hadamard-differentiability, the fact that $B \mapsto 1/B$ is Hadamard differentiable on $\left\{B \in \ell^{\infty}(\mathcal{X}): \inf _{x \in \mathcal{X}}|B(x)|>0\right\}$, Lemma 6 and the Functional Delta Method \citep{kosorok2008introduction}. Therefore, in virtue of Prokhorov's theorem the limiting process is uniformly tight and therefore: $$\left\|\frac{1}{n}\sum_{i=1}^nW_i(k(\cdot,X_i)-\mu_{X^0})\right\|_{\mathcal{G}}=O_p\left(n^{-1 / 2}\right), \quad \left\|\frac{1}{n}\sum_{i=1}^nW_i(l(\cdot,T^*_i)-\mu_{T^0})\right\|_{\mathcal{H}}=O_p\left(n^{-1 / 2}\right) $$

By consistency of real Kaplan-Meier integrals \citep{stute1993}: $\frac{1}{n} \sum_{i=1}^n W_i = o_p(1)$. In addition, the tensor product norm in right hand side of Lemma \ref{lem:ineq} can be seen to be $O_p(n^{-1/2})$ combining our arguments with those in Lemma 5 from \cite{fukumizu07a}. In virtue of Slutsky's theorem, we have just shown the tight uniform bounds \ref{t1}, \ref{t2} we were looking for.
\end{proof}

\begin{Coro} \label{cons}(Consistency) Suppose that Assumptions i.) to vi.) are satisfied. Let $\varepsilon_n>0$ be a regularization constant. Then if $\varepsilon_n \rightarrow 0$ and $n^{1 / 2} \varepsilon_n \rightarrow \infty$ as $n \rightarrow \infty$, we have
$$
\left\|\hat{\mu}_{Y\langle 0 \mid 1\rangle}-\mu_{Y\langle 0 \mid 1\rangle}\right\|_{\mathcal{H}} \rightarrow 0
$$
in probability as $n \rightarrow \infty$.
\end{Coro}

\begin{proof}
See Appendix. 
\end{proof}

Informally, $\alpha$ and $\beta$ in the following result quantify respectively how similar are $F_{X^0}$ and $F_{X^1}$ (the bigger, the more similar) and the smoothness of the map $x \mapsto {\mu}_{T^0 \mid X^0=x}$ (the bigger, the smoother). 

\begin{Coro} \label{final_rate}(Convergence rate) Suppose that Assumptions i.) to vi.) in our paper and Assumption 3, 4 in \cite{cme} hold with $\alpha+ \beta \leq 1$ both non-negative. Let $\varepsilon_n>0$ be a regularization constant. Let $c>0$ be an arbitrary constant, and set $\varepsilon_n=c n^{-1 /(1+\beta+\max (1-\alpha, \alpha))}$. Then we have
$$
\left\|\hat{\mu}_{Y\langle 0 \mid 1\rangle}-\mu_{Y\langle 0 \mid 1\rangle}\right\|_{\mathcal{H}}=O_p\left(n^{-(\alpha+\beta) / 2(1+\beta+\max (1-\alpha, \alpha))}\right) 
$$
\end{Coro}

\begin{proof}
See Appendix. 

\end{proof}

\section{Numerical experiments}

We provide a self-contained simulation study in order to validate the large-sample properties that have been proven in the previous section. 
The underlying model for the simulation case study is

\begin{center}

$\log \tilde{T^0}=X^0_1+X^0_2+\varepsilon$

$\log {C^0}=X^0_1+X^0_2+\varepsilon'$

$\log \tilde{T^1}=2+X^1_1+X^1_2+\omega$

$\log {C^1}=2+X^1_1+X^1_2+\omega'$

\end{center}

$X^0_1$ and $X^0_2$ are independent $\mathcal{N}(0,1)$ random variables while $X^1_1$ and $X^1_2$ are also independent unit variance normal but $X^1_1$ has mean $0.5$. 

$\varepsilon$ and $\varepsilon'$ are $\mathcal{N}(c^0,1)$ and  $\mathcal{N}(0,1)$ respectively with $c^0>0$ controlling the amount of censoring (the bigger $c^0$, the more censoring in the control arm). Analogously, $\omega$ and $\omega'$ are $\mathcal{N}(c^1,1)$ and  $\mathcal{N}(0,1)$ respectively. We have set $c^0=0.2$ and $c^1=0.1$ in order to keep an incomplete information percentage of approximately $75\%$ in both arms through all $B=100$ simulation runs. We replicate the experiment for four different sample sizes $n=100,200,300,500$. We equip both the covariates and response spaces with Gaussian kernel $k\left(y, y^{\prime}\right)=\exp \left(-\left\|{y}-{y}^{\prime}\right\|_2^2 / 2 \sigma^2\right)$. The bandwidth parameter $\sigma$ is chosen via the median heuristic: $\sigma^2=\operatorname{median}\left\{\left\|y_i-y_j\right\|_2^2: i \neq j\right\} / 2$. 

We perform estimation of the causal survival mean embeddings for $B=100$ different simulation runs in the four different sample sizes scenarios. The results are visible in Figure \ref{fig:sim}. Notice the decrease of variability with sample size. We provide in Appendix \ref{ap2} the results of an experiment revealing that our estimator may have a \textit{fast rate}. What is to say, under linear ground dependency between covariates and times, the convergence rate is of stochastic order $\sqrt{n}$, despite this magnitude not being achieved for any non-negative $\alpha$ and $\beta$ in Corollary \ref{final_rate}

\begin{figure}
\centering
\begin{minipage}{.55\textwidth}
  \centering
  \includegraphics[width=\linewidth]{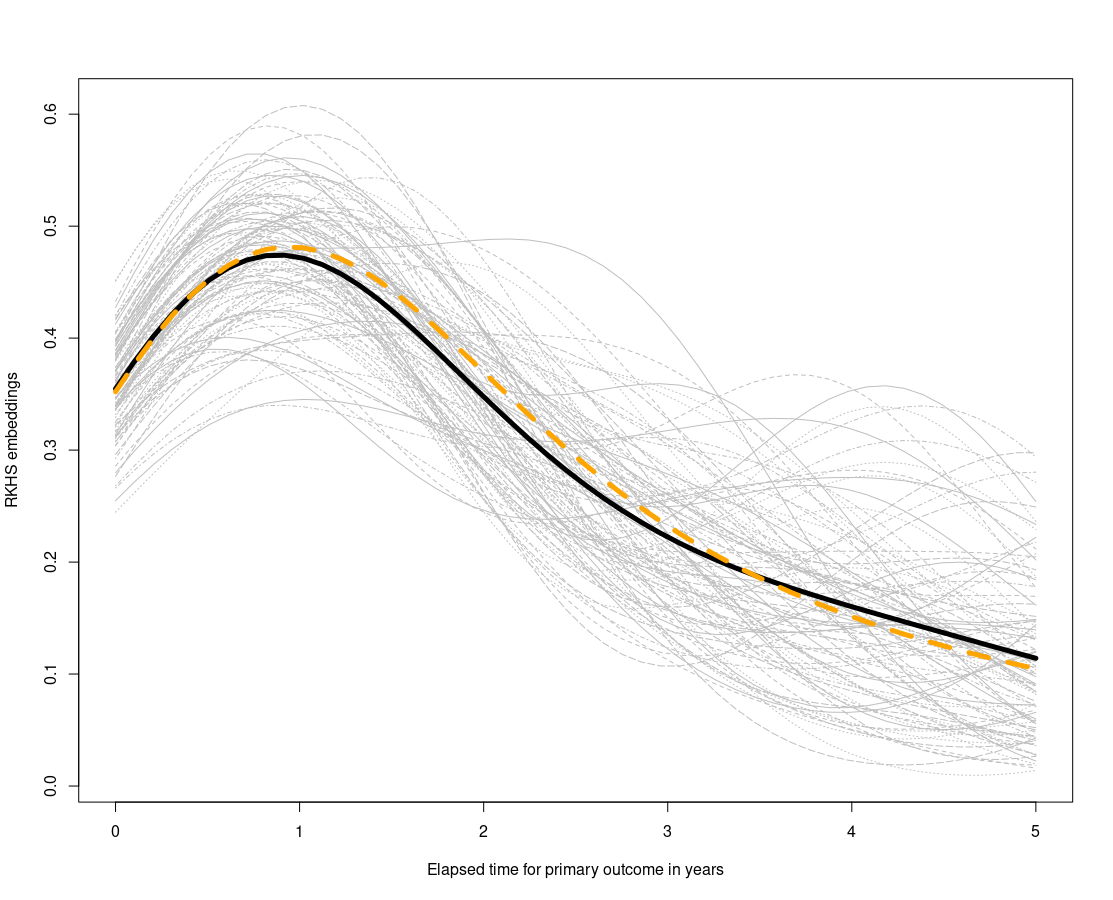}
  \captionof{figure}{$n=100$}
  \label{fig:test1}
\end{minipage}%
\begin{minipage}{.55\textwidth}
  \centering
  \includegraphics[width=\linewidth]{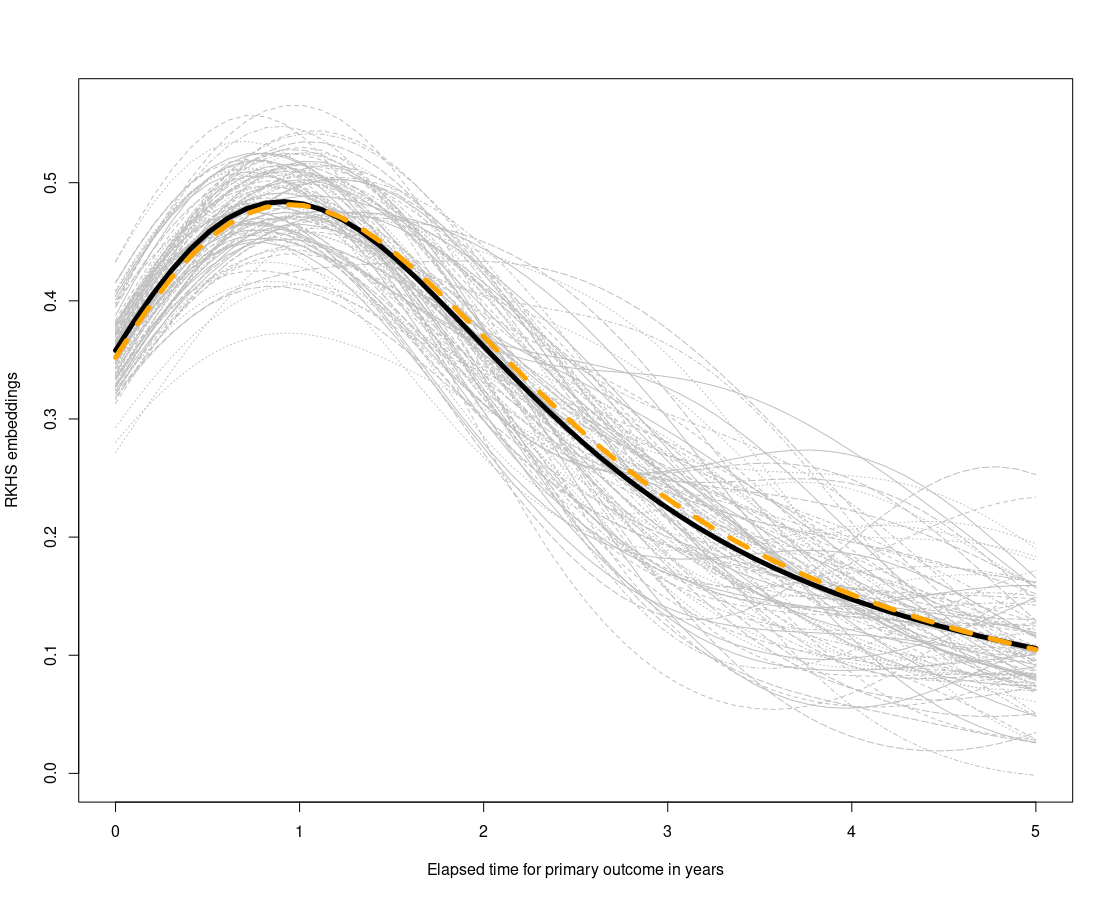}
  \captionof{figure}{$n=200$}
  \label{fig:test2}
\end{minipage}

\begin{minipage}{.55\textwidth}
  \centering
  \includegraphics[width=\linewidth]{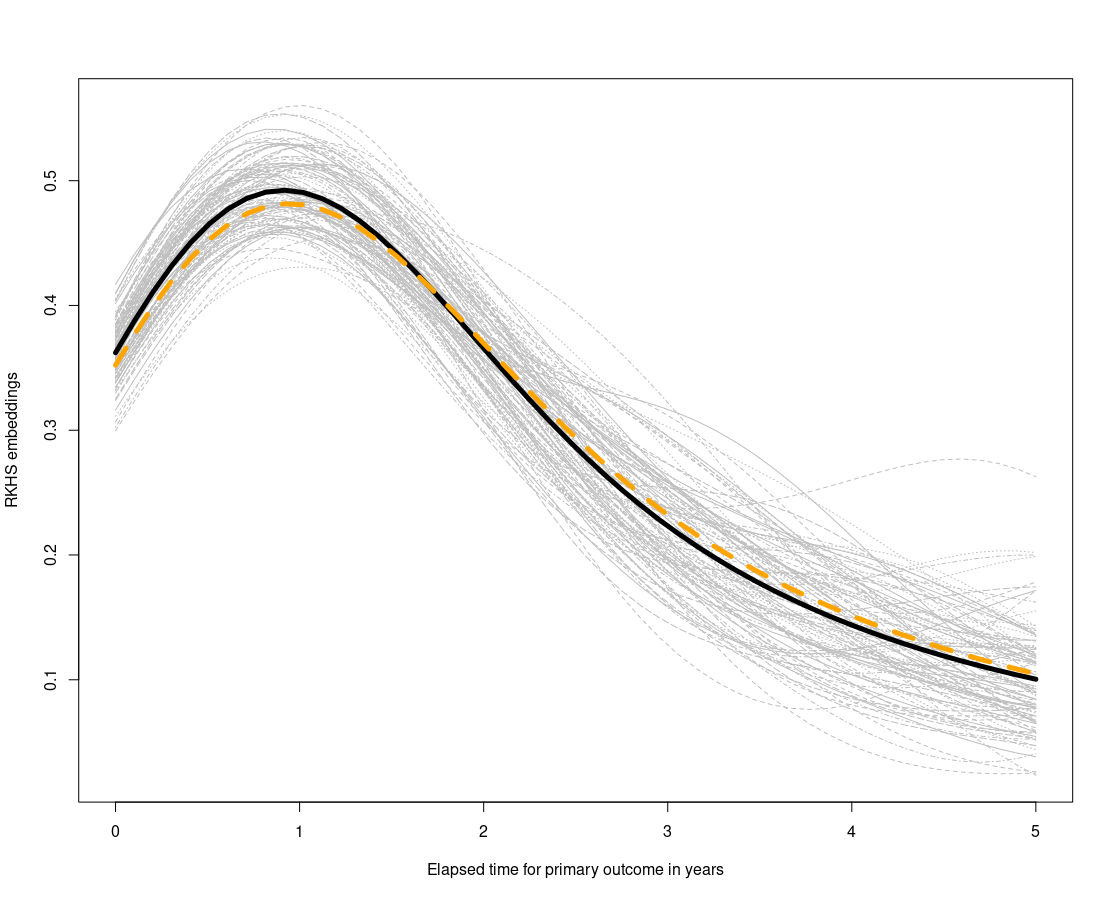}
  \captionof{figure}{$n=300$}
  \label{fig:test1}
\end{minipage}%
\begin{minipage}{.55\textwidth}
  \centering
  \includegraphics[width=\linewidth]{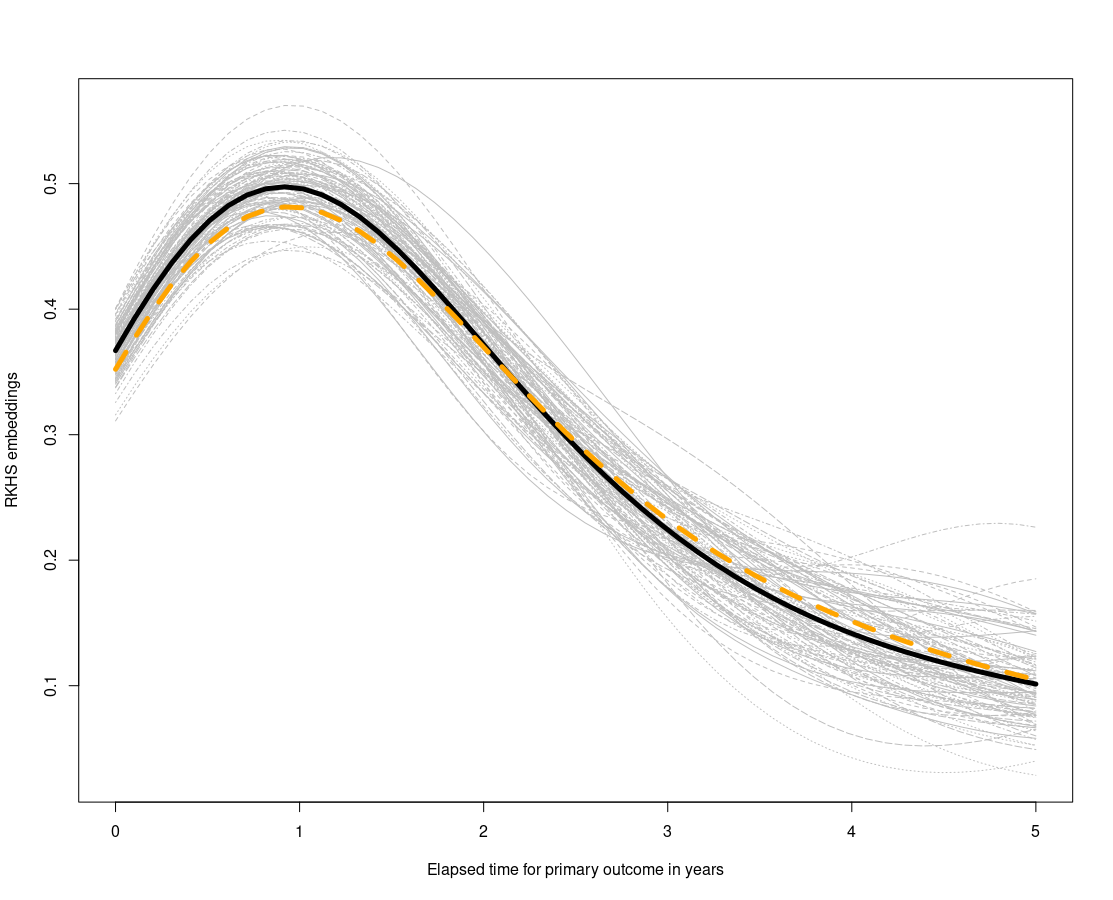}
  \captionof{figure}{$n=500$}
  \label{fig:test2}
\end{minipage}
\caption[Four different simulation scenarios]{The black solid line represents the average of the $B=100$ runs. The dashed yellow line is a numerical approximation of the population counterfactual mean embedding. Each grey line corresponds with one simulation draw. Simulation parameters $c^0$ and $c^1$ were tuned by hand in order to set a censoring percentage of approximately $75 \%$ (on average across simulations, only $25\%$ of information was complete). }
\label{fig:sim}
\end{figure}

\section{Application to SPRINT: a landmark trial in public health}

NIH’s Systolic Blood Pressure Intervention Trial (SPRINT) was conducted to inform the new blood pressure medication guidelines in the US by testing the effects that a lower blood pressure target has on reducing heart disease risk. Observational studies had shown that individuals with lower systolic blood pressure (SBP) levels had fewer complications and deaths due to cardiovascular disease (CVD). Building on this observation, the NIH's Systolic Blood Pressure Intervention Trial (SPRINT) was designed to test the effects of a lower blood pressure target on reducing heart disease risk. Specifically, SPRINT aimed to compare treating high blood pressure to a target SBP goal of less than 120 mmHg against treating to a goal of less than 140 mmHg. 

However, it has been seen in major clinical trials that a reduction of SBP is intimately connected to a reduction of DBP (diastolic blood pressure). Despite this association, it is debated whether low DBP leads to undesirable cardiovascular outcomes, such as a reduction of coronary flow, myocardial infarction, heart failure, or cardiovascular death \citep{dbp1,dbp2,dbp3}. This suggests that intensive systolic blood pressure therapy may result in an excessive reduction of DBP and therefore result in an undesired increase in cardiovascular risk. Nevertheless, SPRINT showed that intensive treatment was clearly associated with a reduced risk of CVD and was even finished early because the results were so convincing \citep{sprintne}. Given the conclusions drawn by SPRINT, the research question is now whether it is possible to decompose the total effect of treatment on the primary outcome into a (natural) direct effect and a (natural) indirect effect through low DBP (induced by the treatment). 

The debate on intensive blood pressure therapy is ongoing. \cite{lee} set out to ascertain whether there is an association between the onset of diastolic hypotension during treatment and negative outcomes. To achieve this, they utilized a conventional Cox PH model, using diastolic blood pressure as a time-varying exposure and adjusting for certain baseline factors. \cite{stensrudathero} aimed to explore whether a formal mediation analysis, utilizing the SPRINT data, could identify whether intensive SBP treatment impacts cardiovascular outcomes via a pathway that involves diastolic blood pressure DBP below 60 mmHg. They claim that \textit{the association between treatment-induced diastolic blood pressure and cardiovascular outcomes suffers from confounding} \citep{stensrudme}.

We illustrate how our methodological contribution manages to perform the desired effect decomposition both across pathways and, importantly, across time thanks to the RKHS formulation.  A consensus answer to the problem would be relevant to the medical community because, as mentioned, SPRINT ultimately informed the new blood pressure guidelines by demonstrating that a lower blood pressure target can significantly reduce heart disease risk.

\begin{figure}[hbt!]
  \centering
  \includegraphics[width=.6\textwidth]{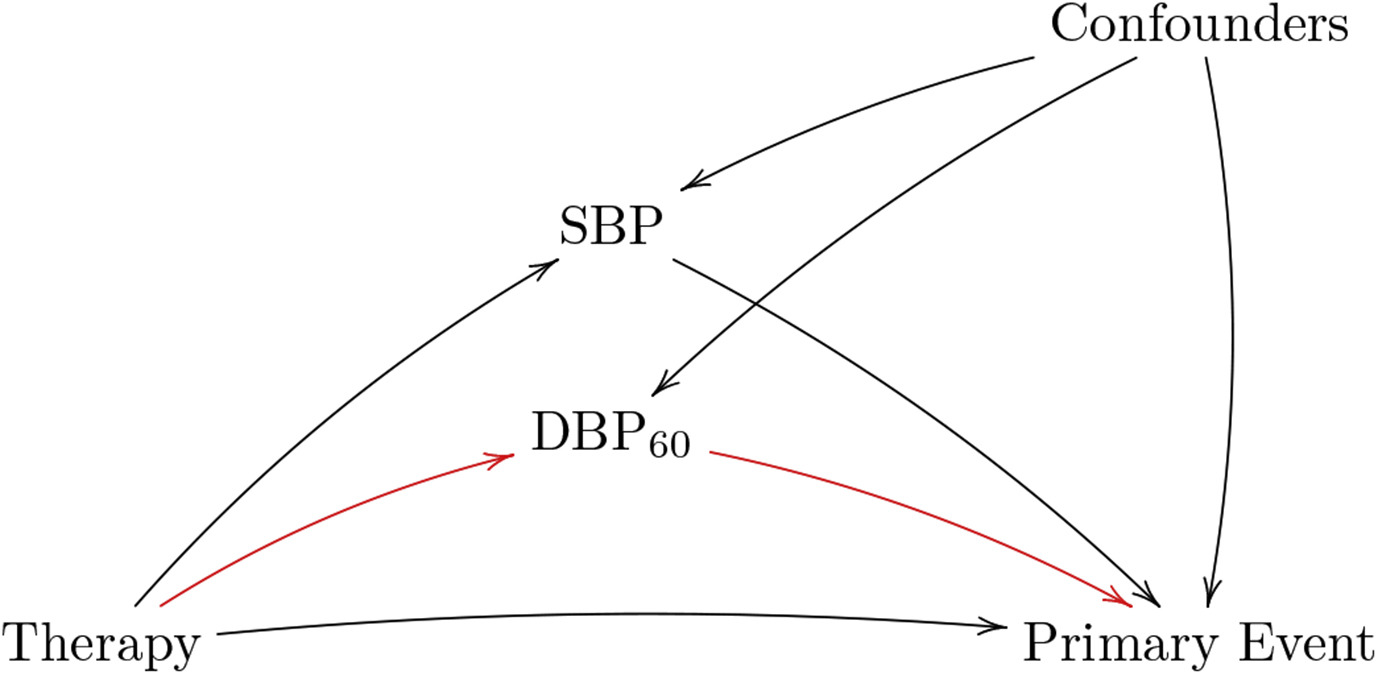} 
  \caption[Geyser data: binned histogram, Silverman's and another
  kernel]
  {DAG depicting underlying causal structure of the medical problem; taken from \cite{stensrudathero}. The primary aim of their investigation was to decompose the total effect of intensive therapy versus standard therapy into two separate pathways: (i) a direct pathway that encompasses all effects not involving a reduction in diastolic blood pressure below 60 mmHg, comprising the advantageous impact of reducing systolic blood pressure, and (ii) an indirect pathway that acts through on-treatment DBP below 60 mmHg and has the potential to be deleterious.}
  \label{fig:geys2}
\end{figure}

\subsection{Description of the dataset}

We conducted our analysis among 2269 participants older than 75 years old who had non-missing values for the covariates. Our response variable \verb|T_PRIMARY| is observed time-to-primary outcome in days, which is a CVD composite endpoint of myocardial infarction, stroke, acute coronary syndrome, acute decompensated heart failure (ADHF), and CVD death. Composite outcomes are postulated to enhance the evaluation of treatment effects on infrequent outcomes, such as mortality in smaller trials, and serve as a convenient means of representing a broader spectrum of beneficial effects resulting from an intervention \citep{cordoba2010definition}. Even considering several events to build the primary endpoint, the percentage of uncensored observations is 11 \% and 7 \% in the control and treatment arms respectively. These high incomplete information percentages render the consideration of censoring mandatory, consitituting a strong motivation factor for the development of our new estimator. 

The treatment indicator for each patient \verb|INTENSIVE| is encoded such that 1 indicates lower SBP target of 120 mmHg and 0 indicates standard treatment (target SBP: 140 mm Hg). The vector of covariates for each patient includes \verb|`DBP.1yr'| (DBP one year after randomisation) and baseline characteristics we want to adjust for:  \verb|`DBP.rz'| DBP at randomization,  \verb|`AGE'|, \verb|`CHR'| Cholesterol mg/dL, \verb|`GLUR'| Glucose mg/dL, \verb|`HDL'| High-Density Lipoprotein ("good") cholesterol direct mg/dL, \verb|`TRR'| Triglycerides, mg/dL, \verb|`UMALCR'| Urine Albumin/Creatinine ratio  \verb|`BMI'| Body mass index kg/m2.

\subsection{Naive analysis of SPRINT}

We might start by stratifying the observations into two groups: one with DBP $\leq 60$ mmHg one year after randomisation (encoded \verb|DBP60=0|) and a group with $> 60$ mmHg one year after randomisation (encoded \verb|DBP60=1DBP60=0|). Then we regress the primary endpoint against the newly created indicator variable using vanilla Cox PH. 

\begin{small}
\begin{verbatim}
> library('survival')

> primary=Surv(t,delta)
> coxdbp60 <- coxph(primary ~ DBP60)
> summary(coxdbp60)

coxph(formula = primary ~ DBP60)

  n= 2269, number of events= 210 

         coef exp(coef) se(coef)     z Pr(>|z|)  
DBP601 0.2823    1.3262   0.1475 1.914   0.0556 .
---
Signif. codes:  0 ‘***’ 0.001 ‘**’ 0.01 ‘*’ 0.05 ‘.’ 0.1 ‘ ’ 1

       exp(coef) exp(-coef) lower .95 upper .95
DBP601     1.326      0.754    0.9933     1.771

Concordance= 0.529  (se = 0.017 )
Likelihood ratio test= 3.54  on 1 df,   p=0.06
Wald test            = 3.66  on 1 df,   p=0.06
Score (logrank) test = 3.69  on 1 df,   p=0.05
\end{verbatim}
\end{small}
The estimates provided by the model fit would confirm the original suspicions of the medical community, stating that low DBP leads to increased cardiovascular risk. This is because the estimate of the hazard ratio \verb|exp(coef)=1.326| $>1$. 

The second step we take is to fit two Kaplan-Meier curves, one for each arm of the SPRINT trial (\verb|INTENSIVE=0| target SBP of 140 mmHg, \verb|INTENSIVE=1| target SBP of 120 mmHg) and produce the plot displayed in Figure \ref{fig:km}. This serves as a quantitative basis for three facts. First, the paradox we are facing becomes empirically confirmed because now treatment defined as SBP lowering intervention seems to be effective (the blue curve therein estimating the survival function of the treatment population is higher after one year). Second, the estimates of the survival functions are crossing. This is a well-known problem in the field of time-to-event analysis \citep{crossing}, directly invalidating the proportional hazards assumption. Third, this would confirm observationally the overall positive results of the SPRINT trial, asserting that intensive SBP control results in cardiovascular benefit.

\begin{figure}[hbt!]
  \centering
  \includegraphics[width=.8\textwidth]{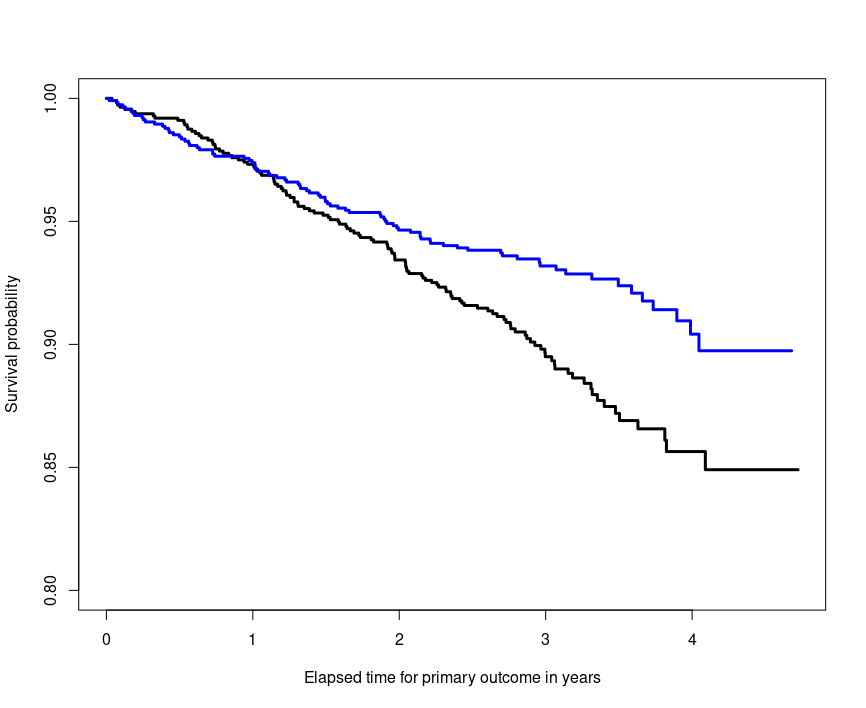} 
  \caption[Geyser data: binned histogram, Silverman's and another
  kernel]
  {Two Kaplan-Meier fits aimed to estimation of $S_{\tilde{T}^1 \mid Z=1}(t)$ in blue and $S_{\tilde{T}^0\mid Z=0}(t)$ in black }
  \label{fig:km}
\end{figure}
\subsection{Conclusions of our analysis of SPRINT}

Our results agree with \citet{stensrudathero}: The increased risk in subjects with diastolic pressure below 60 cannot fully be explained by the intensive treatment itself, but may be due to other factors. A complete description of the results is included in Figure \ref{fig:results}

 \begin{figure}[hbt!]
  \centering
  \includegraphics[width=.8\textwidth]{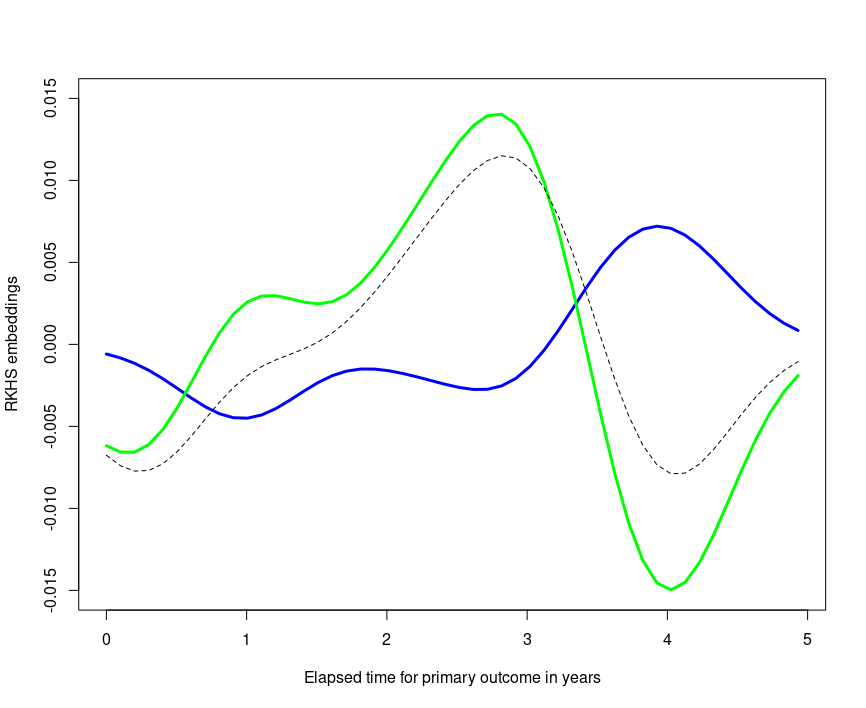} 
  \caption[Geyser data: binned histogram, Silverman's and another
  kernel]
  {We look at Equation \ref{decomp} in the RKHS scale. $\hat \mu_{T\langle 0 \mid 0\rangle} - \hat \mu_{T\langle 0 \mid 1\rangle}$ is represented through the blue line. Similarly, we plot in green $\hat \mu_{T\langle 0 \mid 1\rangle} - \hat \mu_{T\langle 1 \mid 1\rangle}$. An important fact to bear in mind is that (A) in Equation \ref{decomp} is zero if and only if the population counterpart of the blue function is zero, meaning that there is \textit{no} distributional effect on the outcomes arising from the difference in covariate distributions. Likewise, if the green line is zero at the population level then there are no distributional effects for the treated. The dashed line is the sum of both colored functions: the kernel mean embedding of (A) + (B) $= S_{{T}^1}-S_{{T}^0} $ representing the \textit{realized}- not counterfactual- survival probability gain in the intensive treatment arm. The plot can be interpreted as follows: the shift in DBP decrease's distribution across treatment arms has an effect in the opposite direction to the counterfactual treatment effect on the intensive treatment arm; being the latter stronger. During the majority of the study time span (approx. from month 6 until the beginning of the third year), intensively reducing SBP pays off the harmful consequences of the reduction of DBP that comes hand-by-hand. As a consequence, survival probability is increased during this period, which is translated in the dotted line being above zero. Nevertheless, SBP reduction to the lowest target impacts survival negatively in the long term. Interestingly, the inherent reduction in DBP becomes beneficial and tries to compensate for the harmful effect in the long run brought by intensive SBP reduction.  }
  \label{fig:results}
\end{figure}

\section{Discussion}

The main contribution of this paper is the introduction of a novel framework that enables model-free counterfactual inference, opening the doors to many tasks, including counterfactual prediction, hypothesis testing, and clustering analysis. The proposed methods just rely on the Kaplan-Meier estimator \citep{gerds}. While the assumptions that make this possible- i.e., that prevent users from explicitly including covariates in the involved weights- pose a limitation from a practical viewpoint, our method could be locally fitted using the nearest neighbor paradigm \citep{tamas2023recursive} and remain robust to different types of censoring mechanisms.

Moreover, our fundamental approach can be adapted to handle more complex scenarios, such as inverse probability weighting or doubly robust estimators \citep{rubin2007doubly}. The key advantage of our methods is their model-free nature, allowing for learning of complex non-linear relationships between predictors and response variables, given certain smoothness and moment conditions. Many existing models in counterfactual inference are semi-parametric in nature, like the Cox model, which may involve parameters that do not have a causally valid interpretation \citep{martinussen2022causality}. 

The adaptation of such estimators to the fully non-parametric context faces technical difficulties, as seen with the k-NN algorithmand Beran's estimator. However, by adopting the mean embedding toolset, we can create model-free estimators without the technical difficulties. Kernel mean embeddings can be  interpreted as conditional depth bands, proving their usefulness for inferential tasks and other descriptive analyses, as demonstrated in the paper. Additionally, the geometry of kernel mean embeddings allows for a natural interpretation of quantities that are present in the potential outcomes framework, such as the effect of distributional shifts on the covariates. 

From a theoretical standpoint, we discuss the implications of using weights involving the Kaplan-Meier estimator. Roughly speaking, these weights assume independence between survival and censoring times, as well as conditional independence of the censoring indicator and the covariates given the realized times (Assumption v.).

Let us briefly depict the consequences of relaxing these hypotheses. A regular estimator is efficient if it achieves the lowest possible variance among regular estimators, and this optimality notion is established with tools from semiparametric inference \citep{kosorok2008introduction}. Specifically, the Kaplan-Meier integral is asymptotically efficient only under the assumption of independence between survival and censoring times with respect to the covariates \citep{stute1996, laan2003unified}. This is intuitive because the covariate values of the censored times are never observed in empirical estimates. However, if we relax this hypothesis and consider a scenario where $C$ is not independent of $T$ given $Z$, and $\delta$ is not independent of $X$ given $T$ and $Z$, then the resulting estimator will be inefficient; as these assumptions were guaranteeing that the conditional survival distribution of the censoring times $G$ does not depend on the covariates.

To address this issue, we can use a Cox model to estimate ${G}_0(t,x)$. This would be more efficient than using Kaplan-Meier under conscious violation of the previous assumptions, but even this approach will never achieve full efficiency. As per adaptive estimation principle \cite{bickel1993efficient}, a larger censoring model leads to more efficient weights estimation. However, in high-dimensional settings- the scenario we often face when covariates are present in biomedicine, the performance of this method may be poor. This may be potentially alleviated by doubly robust estimators \citep{benkeser2017doubly}. 

The proportional hazards model is the prevailing regression model used in survival analysis. However, a standard Cox analysis does not provide insight into how the effects evolve over time, potentially resulting in loss of valuable information. With the usual Cox analysis, coefficients are typically assumed to remain constant over time, making it challenging to incorporate any deviations from this assumption. There exist a number of alternatives, for instance Aalen's additive regression model \citep{aalen2008survival}. It offers the benefit of permitting covariate effects to vary independently over time. However, Aalen's model performs repeated regressions at each event time, running into instability and overfitting problems when not many events (understood as uncensored observations) are present in the data. Figure \ref{fig:results} illustrates the importance of our estimator as a tool to assess relative risk between treatment arms across time in a natural way without involving time-dependent hazard ratios. All being said, reliably answering inferential questions about time-varying causal effects is a true milestone in contemporary statistics, even reaching areas like Reinforcement Learning \citep{murphy}.

In conclusion, our proposed estimator offers a flexible and powerful tool for estimating counterfactual distributions in observational studies with right-censored data. The model-free nature of our approach makes it applicable to diverse scenarios where traditional methods may be unsuitable. Our estimator can be used in combination with or as an alternative to existing parametric and semiparametric causal survival models, further expanding the range of available options for researchers.

\newpage
\section*{Appendix 1: proofs of auxiliary results}\label{ap1}
\subsection*{Proof of Lemma \ref{l2}}

We have for an arbitrary $G \in \mathcal{F}$
$$\begin{aligned}
&\hat{R}_ {\varepsilon,n}(\hat{F}+G)=\frac{1}{n}\sum_{i=1}^nW_i\left\|h_i-\hat{F}\left(X_i\right) -G\left(X_i\right)\right\|_{\mathcal{H}}^2+\varepsilon\|\hat{F}+G\|_\mathcal{F}^2= \\
&\frac{1}{n}\sum_{i=1}^nW_i\left(\left\|h_i-\hat{F}\left(X_i\right) \right\|_{\mathcal{H}}^2+\left\|G\left(X_i\right)\right\|_{\mathcal{H}}^2-2\langle h_i-\hat{F}(X_i),{G}(X_i) \rangle_{\mathcal{H}}\right)+\varepsilon\left(\|\hat{F}\|_\mathcal{F}^2+\|G\|_\mathcal{F}^2 -2 \langle \hat{F},G\rangle_{\mathcal{F}}\right) = \\
&\hat{R}_ {\varepsilon,n}(\hat{F}) + \frac{1}{n}\sum_{i=1}^nW_i\left(\left\|G\left(X_i\right)\right\|_{\mathcal{H}}^2-2\langle h_i- \hat{F}(X_i),{G}(X_i) \rangle_{\mathcal{H}}\right)+\varepsilon\left(\|G\|_\mathcal{F}^2 +2 \langle\hat{F},G\rangle_{\mathcal{F}}\right) 
\end{aligned}$$ 

Assuming that $\hat{F}$ is a minimizer implies $\hat{R}_ {\varepsilon,n}(\hat{F}) \leq \hat{R}_ {\varepsilon,n}(\hat{F}+G)$ and therefore it is necessary that for all $G \in \mathcal{F}$

$$ \frac{1}{n}\sum_{i=1}^nW_i\langle h_i- \hat{F}(X_i),{G}(X_i) \rangle_{\mathcal{H}}=\varepsilon \langle\hat{F},G\rangle_{\mathcal{F}}$$
Now we try the solution $\hat F=\sum_{i=1}^n \Gamma\left(\cdot, X_i\right)\left(c_i\right) \in \mathcal{F}$ and we use the properties of $\Gamma$ to develop the inner product

$$\langle\hat{F}, G\rangle_{\mathcal{F}}=\langle\sum_{i=1}^n \Gamma\left(\cdot, X_i\right)\left(c_i\right), G\rangle_{\mathcal{F}}=\sum_{i=1}^n\langle c_i, G(X_i)\rangle_{\mathcal{H}}$$

So 

$$ \frac{1}{n}\sum_{i=1}^nW_i\langle h_i- \hat{F}(X_i),{G}(X_i) \rangle_{\mathcal{H}}=\varepsilon \sum_{i=1}^n\langle c_i, G(X_i)\rangle_{\mathcal{H}}$$
Therefore

$$ \sum_{i=1}^n \left(W_i\langle h_i- \hat{F}(X_i),{G}(X_i) \rangle_{\mathcal{H}}-n\varepsilon \langle c_i, G(X_i)\rangle_{\mathcal{H}}\right)=0$$

Now we use again the expression $\hat F=\sum_{i=1}^n \Gamma\left(\cdot, X_i\right)\left(c_i\right)$ to rewrite 

$$\begin{aligned}\langle\hat F(X_i),G(X_ i)\rangle_{\mathcal{H}} =  \sum_{j=1}^n\langle \Gamma(X_i ,X_j)( c_j),G(X_ i)\rangle_{\mathcal{H}}\end{aligned}$$

For the previous identity to be true for all $G \in \mathcal{F}$ it is sufficient that for $1 \leq i \leq n$ the following holds 

$$ W_i( h_i- \sum_{j=1}^n \Gamma(X_i ,X_j) (c_j))-n\varepsilon  c_i=0$$

that can be written as 

$$ W_i h_i=\sum_{j=1}^n W_i \Gamma(X_i ,X_j)(c_j)+n\varepsilon  c_i\delta_{ij}$$

\subsection*{Proof of Lemma \ref{l3}}\

Define $g:=\left(\widehat{\mathcal{C}^*}_{X X}+\varepsilon I\right)^{-1} \hat{\mu}_{X_1}$. \\ Since $\hat{\mu}_{X_1}=\left(\widehat{\mathcal{C}^*}_{X X}+\varepsilon I\right) g=\frac{1}{n} \sum_{j=1}^n W_j k\left(\cdot, X_j\right) g\left(X_j\right)+ \varepsilon g$, \\ we have $\hat{\mu}_{X_1}\left(X_l\right)=\frac{1}{n} \sum_{j=1}^n W_j k\left(X_{l}, X_j\right) g\left(X_j\right)+ \varepsilon g\left(X_{l}\right)=\frac{1}{n}(KW\boldsymbol{g})_{l}+ \varepsilon \boldsymbol{g}_{l}$ for all $l=1, \ldots, n$, where $K \in \mathbb{R}^{n \times n}$ with $K_{i j}=k\left(X_i, X_j\right)$ and $\boldsymbol{g}=\left(g\left(X_1\right), \ldots, g\left(X_n\right)\right)^{\top} \in \mathbb{R}^n$. Therefore $\boldsymbol{\mu}=$ $\frac{1}{n}(KW+n \varepsilon I) \boldsymbol{g}$, where $\boldsymbol{\mu}:=\left(\hat{\mu}_{X_1}\left(X_1\right), \ldots, \hat{\mu}_{X_1}\left(X_n\right)\right)^{\top}=\widetilde{K} 1_m$, where $1_m=(1 / m, \ldots, 1 / m)^{\top}$ and $\widetilde{K} \in \mathbb{R}^{n \times m}$ with $\widetilde{K}_{i j}=k\left(X_i, X_j^{1}\right)$. Thus $\boldsymbol{g}=n(KW+n \varepsilon I)^{-1} \boldsymbol{\mu}$. Lastly, we use the definition of $\widehat{\mathcal{C}^*}_{T X}$ to express $\hat{\mu}_{\langle 0 \mid 1\rangle}=\frac{1}{n} \sum_{i=1}^n W_i\ell\left(\cdot, Y_i\right) g\left(X_i\right)=\sum_{i=1}^n W_i\beta_i \ell\left(\cdot, Y_i\right)$, where $\beta=\left(\beta_1, \ldots, \beta_n\right)^{\top}=$ $n^{-1} \boldsymbol{g}=(KW+n \varepsilon I)^{-1} \boldsymbol{\mu}$, which is the original expression of $\hat{\mu}_{T\langle 0 \mid 1\rangle}$.

\subsection*{Proof of Lemma \ref{lem:ineq}}
\begin{adjustwidth}{-2.5cm}{-1cm}
\begin{align*}
&\left\|\frac{1}{n} \left(\sum_{i=1}^nW_iK_iL_i-\frac{2}{n} \left(\sum_{i=1}^nW_i K_i\right)\left(\sum_{i=1}^nW_iL_i\right) + \frac{1}{n^2}\left(\sum_{i=1}^nW_i K_i\right)\left(\sum_{i=1}^nW_iL_i\right)\left(\sum_{i=1}^nW_i\right)\right)-E[K(X^0) L(T^0)]\right\|_{\mathcal{G}\otimes\mathcal{H}}\\
=&\left\|\frac{1}{n} \sum_{i=1}^nW_iK_iL_i-2 \left(\frac{1}{n}\sum_{i=1}^nW_i K_i\right)\left(\frac{1}{n}\sum_{i=1}^nW_iL_i\right) + \left(\frac{1}{n}\sum_{i=1}^nW_i K_i\right)\left(\frac{1}{n}\sum_{i=1}^nW_iL_i\right)\left(\frac{1}{n}\sum_{i=1}^nW_i\right)-E[K(X^0) L(T^0)]\right\|_{\mathcal{G}\otimes\mathcal{H}} \\
=&\left\|\frac{1}{n} \sum_{i=1}^nW_iK_iL_i-E[K(X^0) L(T^0)]-\left(2 -\frac{1}{n}\sum_{i=1}^nW_i\right)\left(\frac{1}{n}\sum_{i=1}^nW_i K_i\right)\left(\frac{1}{n}\sum_{i=1}^nW_iL_i\right) \right\|_{\mathcal{G}\otimes\mathcal{H}} \\
\leq& \left\|\frac{1}{n} \sum_{i=1}^nW_iK_iL_i-E[K(X^0) L(T^0)]\right\|_{\mathcal{G}\otimes\mathcal{H}} +\left|2 -\frac{1}{n}\sum_{i=1}^nW_i\right|\left\|\left(\frac{1}{n}\sum_{i=1}^nW_i K_i\right)\left(\frac{1}{n}\sum_{i=1}^nW_iL_i\right) \right\|_{\mathcal{G}\otimes\mathcal{H}}
\\
\leq &\left\|\frac{1}{n} \sum_{i=1}^nW_iK_iL_i-E[K(X^0) L(T^0)]\right\|_{\mathcal{G}\otimes\mathcal{H}} +\left|2 -\frac{1}{n}\sum_{i=1}^nW_i\right|\left\|\left(\frac{1}{n}\sum_{i=1}^nW_i K_i\right)\right\|_{\mathcal{G}}\left\|\left(\frac{1}{n}\sum_{i=1}^nW_iL_i\right) \right\|_{\mathcal{H}} 
\end{align*}
\end{adjustwidth}

\subsection*{Proof of Corollary \ref{cons}}
By Triangle's Inequality,

$$\begin{aligned}
\| \widehat{\mathcal{C}}_{T X} & \left(\widehat{\mathcal{C}}_{X X}+\varepsilon_n I\right)^{-1} \hat{\mu}_{X_1}-\mu_{T\langle 0 \mid 1\rangle} \|_{\mathcal{H}} \\
\leq & \left\|\widehat{\mathcal{C}}_{T X}\left(\widehat{\mathcal{C}}_{X X}+\varepsilon_n I\right)^{-1} \hat{\mu}_{X_1}-\mathcal{C}_{TX}\left(\mathcal{C}_{X X}+\varepsilon_n I\right)^{-1} \mu_{X_1}\right\|_{\mathcal{H}} \quad \textrm{(Stochastic error)} \\
& +\left\|\mathcal{C}_{TX}\left(\mathcal{C}_{X X}+\varepsilon_n I\right)^{-1} \mu_{X_1}-\mu_{T\langle 0 \mid 1\rangle}\right\|_{\mathcal{H}}\quad \textrm{(Approximation error)}
\end{aligned}$$

\subsection*{Proof of Corollary \ref{final_rate}}

The proof of Theorem \ref{rate_stoch} has been written with $\alpha=0$ (we can assume so thanks to assumption iii.)). The proof for $\alpha>0$ is straightforward using Lemma 24 in \cite{Muandet_2017} and in this case the rate of the stochastic error is $O_p\left(n^{-1 / 2} \varepsilon_n^{\min (-1+\alpha,-1 / 2)}\right)$. The proof is completed by showing that the rate for the approximation error is $O\left(\varepsilon_n^{(\alpha+\beta) / 2}\right)$ (see E.3 in \cite{cme}).
\newpage
\section{Appendix: empirical check of $\sqrt{n}$ rate under linear truth}\label{ap2}
\begin{small}
 
\end{small}
 \begin{figure}[hbt!]
  \centering
  \includegraphics[width=.6\textwidth]{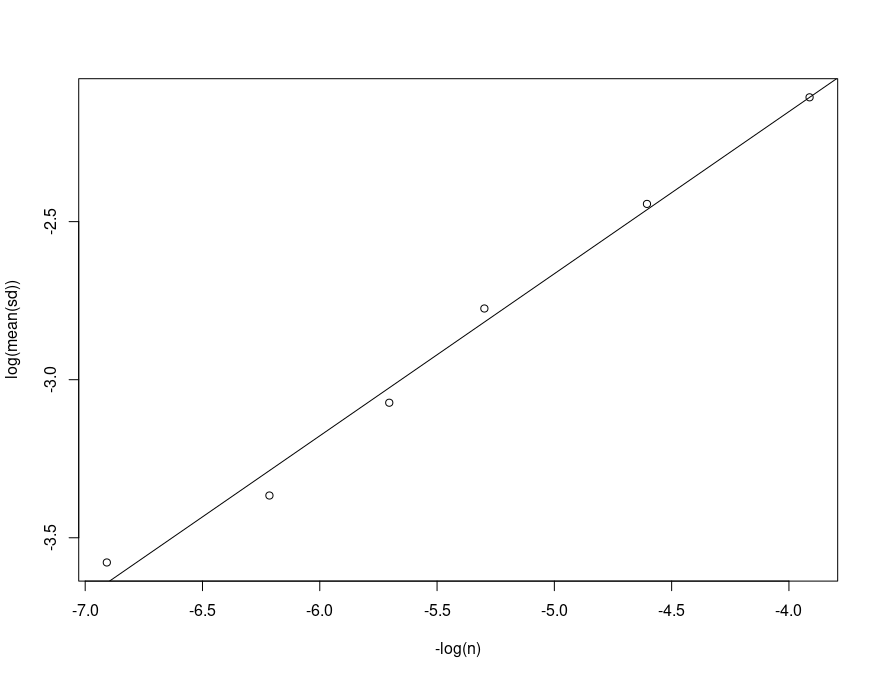} 
  \caption[Geyser data: binned histogram, Silverman's and another
  kernel]
  {Let $V$ be average across time-points of point empirical standard deviation computed through simulations. Assuming that $V=C\cdot n^{-\gamma}$, we can write $\log(V)=\log(C)-\gamma n$. The console output below shows that $\gamma$ is close to 0.5.}
  \label{fig:convrate}
\end{figure}

\begin{Verbatim}[commandchars=\\\{\}]
Call:
lm(formula = lsds ~ mlogn)

Residuals:
         1          2          3          4          5          6 
 8.818e-05  1.834e-02 -7.885e-02  4.292e-02 -4.727e-02  6.477e-02 

Coefficients:
            Estimate Std. Error t value Pr(>|t|)    
(Intercept) -0.10064    0.13888  -0.725    0.509    
mlogn        \textcolor{red}{0.51281}   0.02512  20.418  3.4e-05 ***
---
Signif. codes:  0 ‘***’ 0.001 ‘**’ 0.01 ‘*’ 0.05 ‘.’ 0.1 ‘ ’ 1

Residual standard error: 0.06088 on 4 degrees of freedom
Multiple R-squared:  0.9905,	Adjusted R-squared:  0.9881 
F-statistic: 416.9 on 1 and 4 DF,  p-value: 3.398e-05
\end{Verbatim}

\underline{Note}: This experiment uncovers that our estimator shows an adaptive behaviour: when the underlying model is simulated to be linear, the convergence rate is faster: $n^{-1/2}$.

\newpage

\section{Acknowledgements}

 We would like to express our gratitude to Prof. Peter Bühlmann for his valuable advice during the development of this work. We also extend our gratitude to Prof. Mats Stensrud for his encouragement to use data from the Systolic Blood Pressure Trial (SPRINT). We are grateful for the kind gesture of Prof. Thomas A. Gerds and Prof. Krikamol Muandet, having provided helpful clarifications regarding weak convergence of the estimator. We are thankful to the National Heart, Lung and Blood Institute for providing us with access to this valuable dataset. 
\section{Funding}

C.G.M is supported by the Fundación Barrié via Bolsas de Posgrao no Estranxeiro.


\addtocontents{toc}{\vspace{.5\baselineskip}}
\phantomsection
\addcontentsline{toc}{chapter}{\protect\numberline{}{Bibliography}}




\cleardoublepage
\end{document}